\documentclass[reqno, intlimits, sumlimits]{amsart} 

\usepackage{amsmath, amstext, amssymb, amsthm, bbm, lmodern, enumerate, esint, lipsum, tikz}
\usetikzlibrary{decorations.pathreplacing} %for the curly braces in the new figure

\usepackage{hyperref}
\hypersetup{colorlinks=true,linkcolor=black,citecolor=black,filecolor=black,urlcolor=black}

	\renewcommand{\Re}{\text{Re}}			
	\renewcommand{\Im}{\text{Im}}			
	
    \renewcommand{\tilde}{\widetilde}
    \renewcommand{\theta}{\vartheta} 
    \renewcommand{\phi}{\varphi} 
    \renewcommand{\epsilon}{\varepsilon}
    
	\newcommand{\lebesgue}		{\ensuremath{\lambda\mkern-8mu\lambda}}
	\newcommand				{\eins}			{\mathbbm{1}}   
	\newcommand				{\norm}[1]		{\left\lVert#1\right\rVert}
	\newcommand				{\abs}[1]		{\left\lvert#1\right\rvert}

	\DeclareMathOperator	{\IC}			{\mathbb{C}}
	\DeclareMathOperator	{\IE}			{\mathbb{E}} 
	\DeclareMathOperator	{\IN}			{\mathbb{N}}
	\DeclareMathOperator	{\IP}			{\mathbb{P}}
	
	\DeclareMathOperator	{\IR}			{\mathbb{R}}
	\DeclareMathOperator	{\IZ}			{\mathbb{Z}}
	
	\DeclareMathOperator	{\supp}			{supp}

\theoremstyle{plain}
\newtheorem{thm}			{Theorem}
\newtheorem{lem}	[thm]	{Lemma}
\newtheorem{cor}	[thm]	{Corollary}
\newtheorem{prop}	[thm]	{Proposition}
                     
\theoremstyle{definition}
\newtheorem{defi}	[thm]	{Definition}

\newtheorem{bem}	[thm]	{Remark}

\begin{document}
\title[Rate of convergence for products of random matrices]{Rate of convergence for products of independent non-Hermitian random matrices}
\author{Jonas Jalowy}
 \address{Jonas Jalowy,  Institute for Mathematical Stochastics, University of M\"unster}
 \email{jjalowy@wwu.de}
 \date{\today}
 \thanks{Supported by the German Research Foundation (DFG) through the IRTG 2235}
\begin{abstract}
We study the rate of convergence of the empirical spectral distribution of products of independent non-Hermitian random matrices to the power of the Circular Law. The distance to the deterministic limit distribution will be measured in terms of a uniform Kolmogorov-like distance. First, we prove that for products of Ginibre matrices, the optimal rate is given by $\mathcal O (1/\sqrt n)$, which is attained with overwhelming probability up to a logarithmic correction. Avoiding the edge, the rate of convergence of the mean empirical spectral distribution is even faster. Second, we show that also products of matrices with independent entries attain this optimal rate in the bulk up to a logarithmic factor. In the case of Ginibre matrices, we apply a saddlepoint approximation to a double contour integral representation of the density and in the case of matrices with independent entries we make use of techniques from local laws.
\end{abstract}
 
\maketitle

%%%%%%%%%%%%%%%%%%%%% 

\section{Introduction}\label{sec:Intro}
\noindent The Circular Law states that the empirical spectral distribution of a non-Hermitian random matrix with i.i.d.\ entries converges to the uniform distribution on the complex disc as the size of the matrix tends to infinity. Interestingly, for the product of $m$ independent matrices of such type, the limit distribution will be the $m$-th power of the Circular Law. Here we investigate the question: \emph{How fast does it converge?} The case $m=1$, corresponding to the Circular Law, was studied in \cite{GJ18Rate} already.

We consider the product 
\begin{align*}
\mathbf X=\frac{1}{\sqrt{n^m}}\prod_{q=1}^m X^{(q)}
\end{align*}
of $m$ independent random matrices $X^{(1)},\dots,X^{(m)}$, each of size $n\times n$. For fixed $m\in\IN$, the asymptotic in $n\to\infty$ will be of interest. Its \emph{empirical spectral distribution} is given by
\begin{align*}
 \mu^m_n=\frac{1}{n}\sum_{j=1}^n \delta_{\lambda_j (\mathbf X)},
\end{align*}
where $\delta_\lambda$ are Dirac measures in the eigenvalues $\lambda_j$ of the matrix $\mathbf X$. In this note we are interested in two different classes of random matrices $X^{(q)}$ that appear in the product. 
\begin{defi}
\begin{enumerate}[(i)]
 \item A (complex) \emph{Ginibre matrix} $X$ is a complex non-Hermitian random matrix with independent complex Gaussian entries $X_{ij}\sim\mathcal N_{\IC}(0,1)$.
 \item A non-Hermitian random $n\times n$-matrix $X$ is said to have \emph{independent entries} if $X_{ij}$ are independent complex or real random variables, and in the complex case we additionally assume $\Re X_{ij}$ and $\Im X_{ij}$ to be independent.
\end{enumerate}
\end{defi}
If $X^{(1)},\dots,X^{(m)}$ have independent entries satisfying $\IE X_{ij}^{(q)}=0$ and $\IE\lvert X_{ij}^{(q)}\rvert^2=1$, then the empirical spectral distribution converges weakly to a deterministic probability measure on the complex plane as the matrix size grows. We denote by $\lebesgue$ the $2$-dimensional Lebesgue measure, by $\Rightarrow$ weak convergence of measures and $B_r=B_r(0)$ shall be the open ball of radius $r>0$ centered at $z=0$. In \cite {GT10Products}, G\"otze and Tikhomirov showed that as $n\to\infty$, $\IP$- almost surely we have
\begin{align}\label{eq:LimitProduct}
 \mu_n^{m}\Rightarrow\mu^m_\infty\text{, where }d\mu^m_\infty(z)=\frac {\abs z ^{2/m-2} }{\pi m}\eins_{B_1}(z)d\lebesgue(z)   
\end{align}
is the $m$-th power of the uniform distribution $\mu_\infty=\mu_\infty^1$ on the complex disc, see also \cite{OS11}. The Gaussian case has been treated in \cite{BJW10Products, AB12}, more general models can be found in \cite{KT15, GKT15, bordenave2011, AI15, IpsenKieburg}, for the convergence of the singular values see \cite{AGT10} and furthermore for local results we refer to \cite{nemish,KOV18,N18,GNT17Local,CosProdLin}. 

For $m=1$, we retrieve the well known \emph{Circular Law} $\mu_n=\mu_n^1\Rightarrow\mu_\infty$. In the case of Ginibre matrices this has been discovered much earlier in \cite{Gin65}. We are interested in the rate of convergence, more precisely in a uniform Kolmogorov-type distance over balls 
\begin{align*}
D(\mu_n^m,\mu_\infty^m):=\sup\abs{\mu_n^m(B_R(z_0))-\mu^m_\infty (B_R(z_0)) } 
\end{align*}
as $n\to\infty$, where the supremum runs over all balls $B_R(z_0)\subseteq\IC$. In the sequel, we will also consider the supremum over certain families of balls $B$. Convergence in the distance $D$ coincides with weak convergence in the case of an absolutely continuous limit distribution, see \cite[Lemma A.1]{GJ18Rate}. Using the rotational symmetry of the Circular Law $\mu_\infty$ and the mean empirical spectral distribution $\bar\mu _n =\IE \mu_n$ of the Ginibre ensemble, the following optimal rate of $\mu_n$ for $m=1$ was shown in \cite{GJ18Rate}.

\begin{lem}\label{lem:GinibreRate}
The mean empirical spectral distribution $\bar\mu _n =\IE \mu_n$ of the Ginibre ensemble satisfies
\begin{align}
D(\bar\mu_n,\mu_\infty)\sim \frac 1{\sqrt{2\pi n }}
\end{align}
and for any fixed $\epsilon>0$
\begin{align}\label{eq:GinibreRate2}
\sup_{\substack{B_R(z_0)\subseteq \IC\setminus B_{1+\epsilon}\\ \text{or }B_R(z_0)\subseteq B_{1-\epsilon}}}\abs{\bar\mu_n(B_R(z_0))-\mu_\infty (B_R(z_0)) }\lesssim e^{-n\epsilon^2}.
\end{align}
\end{lem}

Here and in the sequel we denote asymptotic equivalence by $\sim$. We write $\lesssim$, if an inequality holds up to a $n$-independent constant $c>0$ and $A\asymp B$ if $c\abs B\le\abs A\le C\abs B$ for some constants $0<c<C$. These constants $c,C$ may differ in each occurrence. Moreover we abbreviate $\log^ab=(\log b)^a$.

\subsection{The optimal rate of convergence for products of Ginibre matrices}
It is natural to ask for a generalization of Lemma \ref{lem:GinibreRate} to products of $m\geq 1$ independent Ginibre matrices. Notably, our analogous result shows that the optimal rate of convergence does not depend on the number $m$.

\begin{thm}\label{thm:ProductMeanRate}
The mean empirical spectral distribution $\bar\mu^m _n =\IE \mu^m_n$ of the Ginibre ensemble satisfies
\begin{align}\label{eq:ProductMeanRate1}
 \sup_{R>0}\abs{\bar\mu_n^m(B_R)-\mu^m_\infty (B_R) }\asymp \frac{1}{\sqrt{nm}}. 
\end{align}
The following more detailed estimates hold as long as the boundary of the complex disk is avoided
\begin{align}\label{eq:ProductMeanRate2}
 \sup_{R<1-\tfrac m2\sqrt{\log n/ n}}\abs{\bar\mu_n^m(B_R)-\mu^m_\infty (B_R) }&\lesssim\frac{\log^{3/2} n}{n}
 \end{align}
 and uniformly in $ R>1+\sqrt{\log n /n}$ we have
 \begin{align}\label{eq:ProductMeanRate3}
 \abs{\bar\mu_n^m(B_R)-\mu^m_\infty (B_R) }&\lesssim  \exp\left[-n\min((R-1)^2,1)/3\right].
\end{align}
\end{thm}

Theorem \ref{thm:ProductMeanRate} provides the optimal rate of convergence, which however is faster inside and much faster outside of the bulk. The precise constants of the upper and lower bound of \eqref{eq:ProductMeanRate1} can be chosen to be $C=\sqrt{\pi}/\sqrt{2}$ and $c=1/(\sqrt{2\pi })$, coinciding with Lemma \ref{lem:GinibreRate}. The constants in \eqref{eq:ProductMeanRate2} and \eqref{eq:ProductMeanRate3} do depend on $m$. We will also see that the maximal distance in \eqref{eq:ProductMeanRate1} is attained at $R=1$. 

The restriction of Theorem \ref{thm:ProductMeanRate} to centered balls does not affect its character of quantifying weak convergence: Since both $\bar\mu^m _n$ and $\mu_\infty ^m$ have radial symmetric Lebesgue densities, it is easy to see that \eqref{eq:ProductMeanRate1} already implies weak convergence $\bar\mu_n^m\Rightarrow\mu_\infty^m$. In order to remove this restriction one may exploit monotonicity arguments of the radial part of $\bar\mu^m_n$.

While the proof of Lemma \ref{lem:GinibreRate} is an elementary calculation, the proof of Theorem \ref{thm:ProductMeanRate} is more involved and relies on a saddle-point method of a double contour integral representation for the density of $\bar\mu^m_n$, which we will give in Section \ref{sec:Ginibre}. An idea of the proof is given after the contours are defined, see Figure \ref{fig:Contours}. Figure \ref{fig:EVdensity} illustrates the statements of our main results.

\begin{figure}[h]
 \includegraphics[width=.47\textwidth]{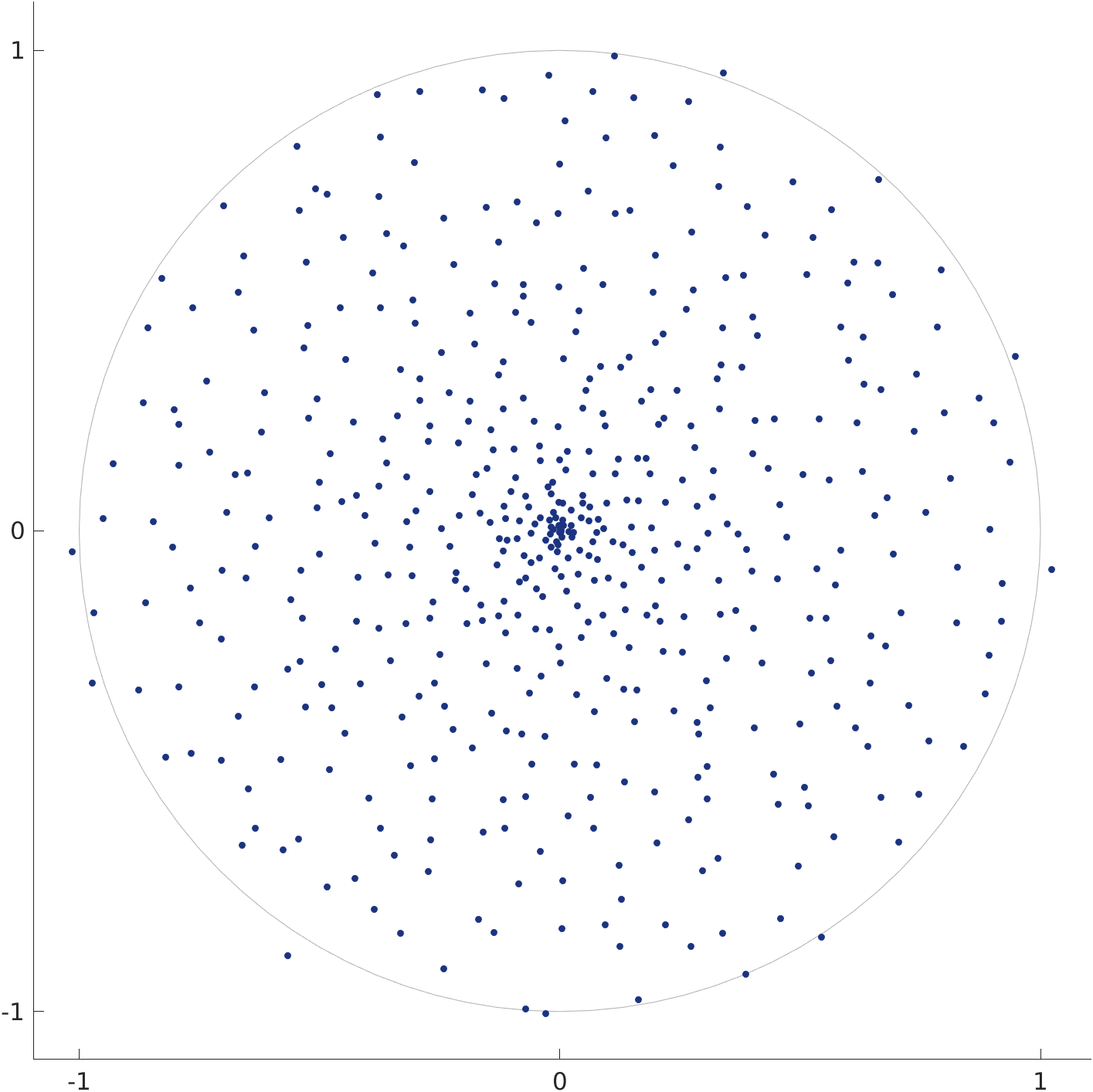}\hfill\includegraphics[width=.47\textwidth]{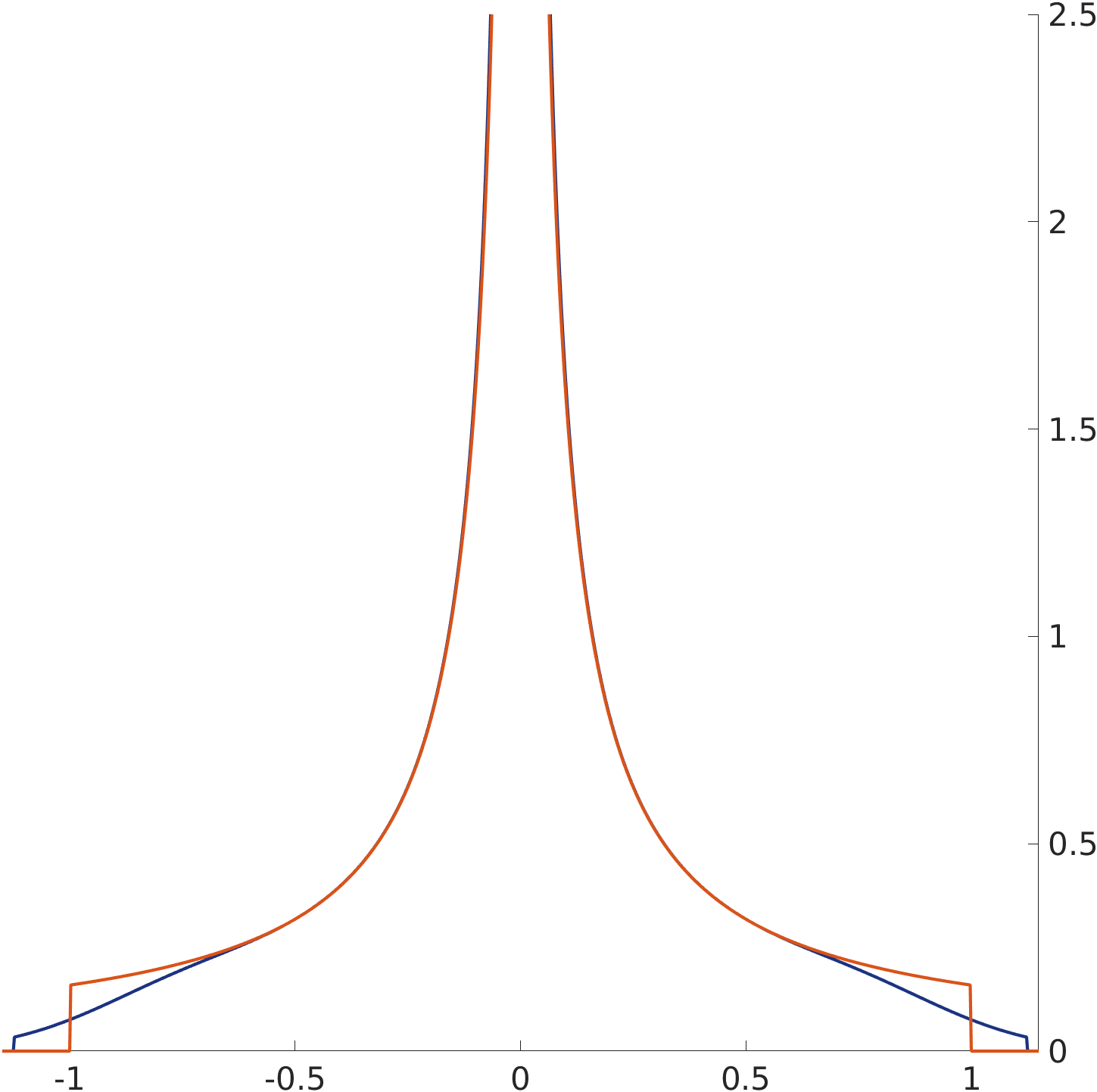}
   \caption{The empirical spectral distribution of the product $\mathbf X$ of $m=2$ Ginibre matrices\newline
  Left: The eigenvalues of a sample for $n=500$ and the unit ball $B_1(0)$ as reference. Theorem \ref{thm:rateofconv} below shows that gaps (like one can see at the top) and clusters do not significantly differ from the limit distribution.\newline
  Right: The radial part of the densities of $\bar\mu^m_n$ for $n=15$ in blue and of the limit distribution $\mu^m_\infty$ in orange. Clearly the rate of convergence in the bulk is faster than close to the edge, illustrating the statement of Theorem \ref{thm:ProductMeanRate}.}\label{fig:EVdensity}
\end{figure}

Pointwise convergence of the density $\rho_n^m$ of $\bar\mu_n$ or $\bar\mu_n^m$ has been also discussed in \cite{AB12,AC18, TV15uni}. In particular Akemann and Burda describe the asymptotic
\begin{align*}
\rho_n^m(z)= \frac{\abs z ^{2/m-2}}{\pi m}\tfrac 12 \text{erfc}\left(\sqrt{\frac{mn}{2}}\frac{\abs z^{2/m}-1}{\abs z ^{1/m}}\right)+o(1)
\end{align*}
for fixed $z$ without specifying the error. Aside from the error $o(1)$, the appearance of erfc$(\cdot)$ also hints at exponential convergence, like in \eqref{eq:GinibreRate2} and \eqref{eq:ProductMeanRate3}. Note that Akemann and Cikovic mention an algebraic rate of convergence for the fixed trace ensemble and conclude that the exponential rate of convergence for Ginibre matrices is rather special. 

Chafa\"\i, Hardy and Ma\"\i da studied invariant $\beta$-ensembles with external potential $V$ instead of matrices with independent entries, see \cite{CHM16}. 
Their result implies a rate of convergence to the limiting measure with density $c\Delta V$ of order $\mathcal O(\sqrt{ \log n / n})$ with respect to the bounded Lipschitz metric and the $1$-Wasserstein distance. Similar questions in this context of $\log$-gases, but for the non-uniform variant of $D$ (the discrepancy) have been addressed in \cite{Serfaty}. Moreover, for determinantal point processes, the fluctuation around the (deterministic) rate has been studied in \cite{FL20}.

\subsection{Non-averaged rate of convergence for products of Ginibre matrices}
Note that we cannot expect an exponentially fast rate of convergence like in Lemma \ref{lem:GinibreRate} for the non-averaged empirical spectral distribution $\mu^m_n$, because it is still affected by the individual eigenvalue fluctuations. In particular the rough lower bound $D(\mu_n^m,\mu_\infty^m)\gtrsim 1/n$ follows from placing a disk of radius $cn^{-1/2}$ contained in $B_1(0)$ such that it does not cover any eigenvalue.
Heuristically, the typical distance of $n$ uniformly distributed eigenvalues in the Circular Law is $n ^{-1/2}$, therefore one may vary $B_R(z_0)$ up to a magnitude of $n^{-1/2}$ without covering a new eigenvalue. Hence we should also expect the non-averaged distance $D(\mu_n^m,\mu_\infty^m)$ to be of order $\mathcal O (n ^{-1/2})$ with overwhelming probability.

\begin{thm}\label{thm:ProductRate}
For products of $m$ Ginibre matrices and any $\epsilon,Q>0$ there exists a constant $c>0$ such that
\begin{align*}
\IP\Big(\sup_{R>0}\abs{\mu_n^m(B_{R} )-\mu_\infty^m( B_{R} )}\le c\sqrt{\frac{\log n}{n}}\Big)\ge 1-n^{-Q}.
\end{align*}
\end{thm}

This rate coincides exactly with the Wasserstein rate obtained in \cite{CHM16}. The proof makes use of the determinantal structure of the eigenvalues of Gaussian matrices in order to apply Bernstein's inequality and will be given in the end of Section \ref{sec:Ginibre}.
%The logarithmic factor is expected to be non-optimal and might be removed by controlling the variance term in \eqref{eq:Bernstein} below, but this will not be pursued here.

\subsection{Rate of convergence for products of non-Gaussian matrices}

Based on the ideas of \cite{GJ18Rate}, we will also prove a rate of convergence result for products of matrices with independent entries.
\begin{defi}[Condition (C)]
We say $\mathbf X=X^{(1)}\cdots X^{(m)}/\sqrt{n^m}$ satisfies condition (C) if the matrices $X^{(q)}$, $q=1,\dots, m$, have jointly independent entries and satisfy 
$$\abs{\IE X_{ij}^{(q)}}\le n^{-1-\epsilon}\text{ and }\abs{1-\IE \abs{X_{ij}^{(q)}}^2}\le n^{-1-\epsilon}$$ for some $\epsilon>0$ independent of $n$ and furthermore 
 $$\max_{i,j,q,n} \IE \abs{X_{ij}^{(q)}}^{4+\delta}<\infty$$ for some $\delta>0$.
\end{defi}

In Section \ref{sec:iidMatrices}, we will show

\begin{thm}\label{thm:rateofconv}
If condition (C) holds, then for every $m\in\IN$ and $\tau,Q>0$ there exists a constant $c>0$ such that
\begin{align*}
\IP\Big(\sup_{B}\abs{(\mu^m_n-\mu^m_\infty)(B)}\leq c\hspace{1pt} h_m(n)\Big)\geq 1-n^{-Q},
\end{align*}
where the supremum runs over balls $B \subseteq B_{1-\tau}\cup B_{1+\tau}^c$ and the asymptotic error is given by
\begin{align*}
h_m(n)=\begin{cases}
        n^{-1/2}\log^2n&\text{ for }m=1,\\
        n^{-1/2}\log^3n&\text{ for }m=2,\\
        n^{-2/(m+2)}\log^{8/(m+2)}n&\text{ for }m\ge 3.\\
        \end{cases}
\end{align*}
\end{thm}

In the proof of Theorem \ref{thm:rateofconv} we will see that the $m$-dependent term is only visible for balls touching the origin and else only the error $h_1$ remains for all $m$. To make the statement more comprehensible when comparing with Theorem \ref{thm:ProductMeanRate}, we also state the following result.

\begin{cor}\label{cor:rateofconv}
If condition (C) holds, then for every $\tau,Q>0$ we have
\begin{align*}
\IP\Big(\sup_{B}\abs{(\mu^m_n-\mu^m_\infty)(B)}\lesssim\frac{\log^2 n}{\sqrt n}\Big)\geq 1-n^{-Q},
\end{align*}
where the supremum runs over all balls $B$ such that $\partial B_R(z_0)\subseteq B_{1+\tau}^c\cup B_{1-\tau}\setminus B_\tau$ avoids the edge and the origin.
\end{cor}
 We already know from the previous discussion that the optimal rate is given by $\mathcal O (1/\sqrt n)$. Corollary \ref{cor:rateofconv} shows that this rate is (almost) universal in the sense that it continues to hold for matrices with independent entries, if edge and origin are avoided. A weaker rate of convergence for $m=1$ has already been established in \cite[\S 2.2]{GJ18Rate}, where a comparison with similar results can be found as well. We would like to point out a subtle difference between Theorem \ref{thm:rateofconv} and the Local Law in \cite{GNT17Local}; the latter compares an integral over a smooth function with the limiting distribution on a shrinking support for a \emph{fixed} point $z_0$, while the former allows to choose the ``worst ball`` $B_R(z_0)$ depending on the random sample of eigenvalues.
 
\section{Product of Ginibre matrices}\label{sec:Ginibre}

We start with the following double contour integral representation for the density of $\mu^m_n$ that is essential for Theorem \ref{thm:ProductMeanRate}.

\begin{lem}\label{lem:DoubleContour}
The density of $\bar\mu^m _n$ satisfies
\begin{align*}
 \rho^m_n(z)=\frac 1{n(2\pi i)^2 } \oint_{\gamma}\int_{\frac12-i\infty}^{\frac12+i\infty}\left(\frac{\Gamma(s)}{\Gamma(t)}\right)^m n^{m(t-s)}\abs z ^{2(t-s-1)}\cot(\pi t)dsdt,
\end{align*}
where $\gamma$ is any closed contour that encircles the numbers $1,\dots,n$ counter clockwise and no natural number greater than $n$.
\end{lem}
In \cite{KZ14SV}, a similar double contour integral representation for the correlation kernel of the singular values of $\mathbf{X}$ was derived. This was used in \cite{Liu} to prove bulk universality for singular values of products of independent Ginibre matrices. In general, double contour integrals like in Lemma \ref{lem:DoubleContour} above appear regularly in the theory of products of random matrices, e.g. \cite{KZ14SV, FW, KKS}.

In the sequel we will make use of the \emph{Meijer G-function}, which for $\xi\in\IC\setminus \{0\}$ is defined as the Mellin inverse of products of Gamma functions
\begin{align}\label{eq:MeijerGdef}
 G^{m,n}_{p,q}\left(\genfrac{}{}{0pt}{}{a_1,\dots,a_p}{b_1,\dots,b_q}\Big\vert \xi\right)=\frac{1}{2\pi i}\int_L\frac{\prod_{j=1}^m\Gamma(b_j-t)\prod_{j=1}^n\Gamma(1-a_j+t)}{\prod_{j=m+1}^q\Gamma(1-b_j+t)\prod_{j=n+1}^p\Gamma(a_j-t)}\xi^tdt,
\end{align}
where $0\le m\le q$, $0\le n \le p$, $a_k-b_j\not\in \IN$ for $k=1,\dots,n$ and $j=1,\dots ,m$.
The contour $L$ goes from $-i\infty$ to $i\infty$, but can be chosen arbitrarily as long as the poles of $\Gamma (b_j-t)$ are on the right hand side of the path and the poles of $\Gamma(1-a_j-t)$ are on the left hand side. Since the Gamma function is the Mellin transform of the exponential function, we have for instance $ G^{1,0}_{0,1}\big(\genfrac{}{}{0pt}{}{-}{0}\big\vert \xi ^2\big)=\exp(-\xi^2)$.

In particular the density of $\bar\mu_n^m$ is given by
\begin{align}\label{eq:DensityProducts}
 \rho^m_n(z)=n^{m-1}\sum_{k=0}^{n-1}\frac{n^{mk}\abs z ^{2k}}{\pi(k!)^m}G^{m0}_{0m}\left(\genfrac{}{}{0pt}{}{-}{0}\Big\vert n^m\abs z ^2\right),
\end{align}
see \cite{AB12} and compare to the case $m=1$, where $G^{10}_{01}\big(\genfrac{}{}{0pt}{}{-}{0}\big\vert n\abs z ^2\big)=e^{-n\abs z ^2}$. 

We would like to point out that the Coulomb gas picture partially breaks down here, since $V=-\log G^{m,0}_{0,m}\left(\genfrac{}{}{0pt}{}{-}{0}\Big\vert \abs{\cdot} ^2\right)$ is not admissible (according to \cite{Serfaty}). In particular $V$ is unbounded from below in its singularity and hence $\Delta V\sim \mu_\infty ^m$ fails for finite $n$.

\begin{bem}
The viewpoint of studying products of $m$ matrices and definition \eqref{eq:DensityProducts} of $\rho^m_n$ makes sense for $m\in\IN$ only. However the representation of Lemma \ref{lem:DoubleContour} makes sense for arbitrary $m>1, m\in\IR$. Furthermore, as we can see from the proof of Theorem \ref{thm:ProductMeanRate}, its statements \eqref{eq:ProductMeanRate1} and \eqref{eq:ProductMeanRate3} remain true for real $m>1$, as well as \eqref{eq:ProductMeanRate2} for real $m\ge 2$.
\end{bem}

\begin{bem}
Since the constant in \eqref{eq:ProductMeanRate1} of Theorem \ref{thm:ProductMeanRate} is explicit in $m$, it is possible to consider the double scaling limit and let $m=m(n)\to\infty$. A careful inspection of the proof reveals that at most $m=o(n)$ is possible. In this case, the rate will be faster, depending on $m$, and in particular by setting $m=n$, we expect a rate  $\sup_{R>0} \abs{\bar\mu_n^n(B_R) -\mu_{\infty}^n(B_R) } \asymp \frac{1}{n}$. 
On the other hand, $\mu_\infty^m$ converges weakly to $\mu^{\infty}_\infty=\delta_0$ as $m\to\infty$ and for fixed $R>0$ we see that $\delta_0(B_R)-\mu_{\infty}^m(B_R)=1-R^{2/m}/(2\pi)=\mathcal O(1/m)$. Note that $m\sim n$ is also the critical scaling of Lyapunov exponents between deterministic and GUE statistics, see \cite{ABK,LiuWang}.
\end{bem}

\begin{proof}[Proof of Lemma \ref{lem:DoubleContour}]
For the contour $L$ of the Meijer G-function in \eqref{eq:DensityProducts}, we choose the straight vertical line $L=[-1/2-i\infty,-1/2+i\infty]$ that after a simple change of variables $-t=s$ leads to
\begin{align}\label{eq:MeijerGm00m}
 G^{m,0}_{0,m}\left(\genfrac{}{}{0pt}{}{-}{0}\Big\vert n^m\abs z ^2\right)=\frac{1}{2\pi i}\int_{\frac12-i\infty}^{\frac12+i\infty}\Gamma(s)^m(n^m\abs z ^2)^{-s}ds
\end{align}
The remaining part of \eqref{eq:DensityProducts} is the kernel of the (monic) orthogonal polynomials with respect to the Meijer-G-weight. It can be rewritten with the help of the residue theorem. For any closed curve $\gamma$ encircling the numbers $1,\dots,n$ but no natural number greater than $n$, we have
\begin{align*}
\frac{1}{2\pi i}&\oint_{\gamma} \frac{(n^m\abs z^2)^{t-1}}{\Gamma(t)^m}\cot(\pi t)dt=\sum_{k=0}^{n-1}\frac{n^{mk}\abs z ^{2k}}{\pi(k!)^m},
\end{align*}
since the integrand is holomorphic except for its simple poles in $\IN$ with residues $1/\pi$ of each cotangent. Combining both contour integrals proves the claim.
\end{proof}

Asymptotic expansions of $G^{m,0}_{0,m}$, like \cite[\S 5.9.1.]{Luke}, together with heuristics for the hypergeometric kernel give rise to pointwise limits in \cite{AB12}. A rigorous estimation of the error bound uniformly in $z$ seems to be absent in the literature so far. 
%Even the asymptotic formulas for $G^{m,0}_{0,m}$ are valid only for fixed $z$, hence integration around $z=0$ is impossible. Integrating $\rho^m_n$ first leads to $n$-dependent coefficients $a_j$ or $b_j$ in the Definition \eqref{eq:MeijerGdef} of the Meijer G-function, but its asymptotics, e.g. from \cite{Luke}, are not uniform in these coefficients. 
Hence it is reasonable to study the problem by a direct analysis.

\begin{proof}[Proof of Theorem \ref{thm:ProductMeanRate}]
By Lemma \ref{lem:GinibreRate}, it is sufficient to consider $m\ge 2$. For $R>1$ we have $\abs{(\bar\mu_n^m-\mu_\infty^m)(B_R)}<\abs{(\bar\mu_n^m-\mu_\infty^m)(B_1)}$, since $\supp(\mu_\infty^m)=B_1(0)$. Throughout the proof we assume 
\begin{align}\label{eq:AssR}
\log^{3m/4}n/n^{m/2} \le R \le 1
\end{align}
since for smaller values of $R$ it holds
\begin{align*}
\abs{(\bar\mu_n^m-\mu_\infty^m)(B_R(0))}\le\abs{(\bar\mu_n^m-\mu_\infty^m)(B_{\log^{3m/4}n/n^{m/2}}(0))}+\mathcal O(\log^{3/2} n/n),
\end{align*}
due to $\mu^m_\infty(B_R)=R^{2/m}$. We first use spherical symmetry of $\rho^m_n$ and Lemma \ref{lem:DoubleContour} in order to calculate 
\begin{align}\label{eq:DCmu}
 \bar\mu_n^m(B_R)&=\int_0^{R^2}\pi\rho^m_n(\sqrt{r })dr\nonumber\\
 &=\frac \pi{n(2\pi i)^2 } \oint_{\gamma}\int_{\frac12-i\infty}^{\frac12+i\infty}\left(\frac{\Gamma(s)}{\Gamma(t)}\right)^m \frac{(n^mR^2)^{t-s}}{t-s}\cot(\pi t)dsdt.
\end{align}
This holds in the case where $s$ and $t$ have distance bounded from below, which is what we will choose in the following.
We will now show that shifting the vertical contour in Lemma \ref{lem:DoubleContour} to $L=[\eta-i\infty, \eta+i\infty]$ for another real part $\eta\geq1/2$, $\eta\neq 1,\dots,n$, produces an additional term. Cauchy's integral formula implies
\begin{align*}
\frac \pi{(2\pi i)^2 }\left(\int_{\frac12-i\infty}^{\frac12+i\infty}-\int_L\right) \left(\frac{\Gamma(s)}{\Gamma(t)}\right)^m \frac{(n^mR^2)^{t-s}}{t-s}ds =\frac \pi{2\pi i }\eins_{(1/2,\eta)}(\Re (t)).
\end{align*}
We temporarily split $\gamma$ into two parts $\gamma_l$ and $\gamma_r$ such that $\gamma_l$ encircles $\{1,\dots, \lfloor\eta\rfloor\wedge n\}$ and $\gamma_r$ encircles $\{\lceil\eta\rceil,\dots,n\}$. Soon we will make the path of $\gamma$ more explicit. As in \eqref{eq:DCmu}, continuing the integration of the right hand side of the last equation multiplied with $\cot( \pi t)$ over $\gamma_l\cup\gamma_r$ yields
\begin{align}\label{eq:roundingError}
\frac \pi{2\pi i }\oint_{\gamma_l}\cot(\pi t)dt=\lfloor\eta\rfloor\wedge n,
\end{align}
hence we conclude
\begin{align*}
 \bar\mu_n^m(B_R)&=\frac \pi{n(2\pi i)^2 } \oint_{\gamma}\int_L\left(\frac{\Gamma(s)}{\Gamma(t)}\right)^m \frac{(n^mR^2)^{t-s}}{t-s}\cot(\pi t)dsdt+\frac{\lfloor\eta\rfloor}n\wedge 1.
 \end{align*}
Choosing $\eta=\lfloor nR^{2/m}\rfloor+1/2$ we see that the second term is $\mathcal O (1/n)$ close to $\mu^m_\infty(B_R)=R^{2/m}\wedge 1$. Moreover, by Cauchy's integral formula, we may artificially add the removed part $\gamma-\gamma_l-\gamma_r$ again as long as $\gamma$ is symmetric around the $x$-axis. Let $\gamma$ be the rectangular contour connecting the vertices $3/4-i$, $n+1/4-i$, $n+1/4+i$ and $3/4+i$. This ensures a constant distance to the singularities of the cotangent. The scaled version $\tilde\gamma=\gamma/(R^{2/m}n)$ is illustrated below in Figure \ref{fig:Contours}. Furthermore note that the integral exists as we will explicitly show below, see \eqref{eq:ContourIntegrability}. 
Recall Stirling's formula for the Gamma function
\begin{align*}
 \log\Gamma(z)=(z-1/2)\log z-z+\frac12\log 2\pi+\mathcal O(1/\Re z),
\end{align*}
which holds uniformly for $\Re z\geq 1/2$, cf. for instance \cite[p.249]{Whit96}.
Thus, we have 
\begin{align}\label{eq:Stirling}
\log\left(\frac{\Gamma(s)^m}{(n^mR^2)^{s}}\right)=m\left(s\left(\log\left(\frac {s}{nR^{2/m}}\right) -1\right) +\frac12\log\left(\frac{2\pi} {s}\right)\right)+\mathcal O(1/ {\Re(s)}),
\end{align}
where for all $s\in  L$ (and analogously for $t\in\gamma$) the error term is at most $\mathcal O (1)$. We rescale the integration by $nR^{2/m}$ and denote $\tilde \gamma =\gamma/(nR^{2/m})$, $\tilde L=L/(nR^{2/m})$ as well as \begin{align*}
F(z)=z\log z-z
\end{align*}
to obtain
\begin{align}\label{eq:DCmu2}
  \bar\mu_n^m(B_R)-\frac{\lceil nR^{2/m}\rceil}n&\nonumber\\
  =\frac {\pi R^{2/m}}{(2\pi i)^2 } \oint_{\tilde\gamma}\int_{\tilde L}&\left(\frac{\Gamma(nR^{2/m}s)}{\Gamma(nR^{2/m}t)}\right)^m \frac{(n^mR^2)^{(nR^{2/m})(t-s)}}{t-s}\cot(\pi nR^{2/m} t)dsdt\nonumber\\  
  =\frac {\pi }{(2\pi i)^2 } \oint_{\tilde\gamma}\int_{\tilde L}&\exp\left[ nmR^{2/m}\left(F(s)-F(t)\right)\right]\left(\frac ts\right)^{m/2}\frac{\cot(\pi nR^{2/m} t)}{t-s}\\
  &\cdot\left(R^{2/m}+\mathcal O \left(\tfrac1{n\Re(s)}\right)+\mathcal O \left(\tfrac1{n\Re(t)}\right)\right)dsdt.\nonumber
\end{align}

\begin{figure}[h]
 \begin{tikzpicture}
\filldraw[color=gray!80, fill=gray!30, thick] (3.5,-2.2) rectangle (7.9,2.2);
\draw[gray, line width=0.5pt] (-0.8,0) -- (2.5,0);
\draw[gray, line width=0.5pt, dotted] (2.5,0) -- (4,0);
\draw[gray, line width=0.5pt,->] (3.5,0) -- (11,0);
\draw[gray, line width=0.5pt,->] (0,-3.5) -- (0,3.5);
\foreach \x in {0,1,2,4,5,6,7,8,9,10}
\filldraw [gray] (\x,0) circle (1pt);
\filldraw [gray] (0,1) circle (1pt);
\filldraw [gray] (0,2.2) circle (1pt);
\node[below left] at (0,0) {0};
\node[below] at (1,0) {$\tfrac1{nR^{\frac2m}}$};
\node[below] at (2,0) {$\tfrac2{nR^{\frac2m}}$};
\node[below] at (8,0) {$\tfrac{n-1}{nR^{\frac2m}}$};
\node[below] at (9,0) {$\tfrac{1}{R^{\frac2m}}$};
\node[below] at (10,0) {$\tfrac{n+1}{nR^{\frac2m}}$};
\draw[line width=0.8pt] (9.25,1)--(4,1);
\draw[line width=0.8pt, dotted] (2.5,1)--(4,1);
\draw[line width=0.8pt,->] (2.5,1)--(0.75,1)--(0.75,-1);
\draw[line width=0.8pt] (0.75,-1)-- (2.5,-1);
\draw[line width=0.8pt, dotted] (2.5,-1)--(4,-1);
\draw[line width=0.8pt,->] (4,-1)--(9.25,-1)--(9.25,1);
\draw[line width=1.6pt,->] (7.9,1)--(3.5,1);
\draw[line width=1.6pt,->] (3.5,-1)--(7.9,-1);
\draw[gray, dotted] (0,1)--(0.75,1);
\draw[gray, dotted] (0,2.2)--(5.5,2.2);
\draw[line width=0.8pt,->] (5.5,-3.25)--(5.5,3.25);
\draw[line width=1.6pt,->] (5.5,-2.2)--(5.5,2.2);
\filldraw (5.5,0) circle (2pt);
\node[left] at (0.1,1) {$\frac{i}{nR^{\frac{2}{m}}}$};
\node[left] at (0.1,2.2) {$i\delta_n$};
\node[right] at (5.5,1.7) {$L_{loc}$};
\node at (5.8,0) {${\scriptstyle \times}$};
\node[below] at (5.8,0) {\textbf{1}};
\node[below] at (4,1) {$\gamma_{loc}$};
\node[above] at (7,-1) {$\gamma_{loc}$};
\node[below] at (1.6,-1) {$t\in\tilde\gamma$};
\node[left] at (5.5,-2.8) {$s\in\tilde L$};
\node[above left] at (5.63,0) {$\frac\eta {nR^{\frac2m}}$};
\node[below] at (10.8,0) {$\Re$};
\node[left] at (0,3.3) {$\Im$};
\node[right] at (9.25,0.5) {$\gamma_{vert}$};
\node at (7.8,-2.5) {$Q_{\delta_n}(1)$};
\end{tikzpicture}
\caption{The scaled contours of integration and their local parts in thicker lines. As we will see later, the main contribution comes from $\gamma_{loc}$ that is in a box $\delta_n$-close to the saddle point at $z=1$. If $R> 1$, there is no $\gamma_{loc}$ and the integral vanishes exponentially fast (depending on the distance $\abs{R-1}$). If $R< 1$, then both horizontal contours of $\gamma_{loc}$ will cancel, because of their symmetry. In this case we will obtain a rate of convergence of microscopic order $1/n$ due to the discrete nature of the residues. The maximal rate will be attained for $R=1$, where the integrals do not cancel, yet the vertical part $\gamma_{vert}$ is small enough.}
\label{fig:Contours}
\end{figure}
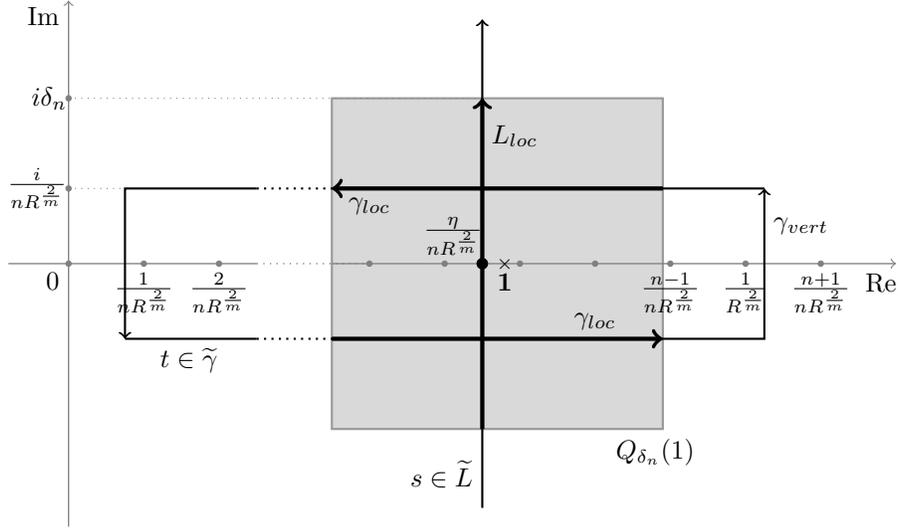

Observe that $\mathcal O (1/(n\Re(t)))=\mathcal O (R^{2/m})$ and $\mathcal O (1/(n\Re(s)))=\mathcal O (1/n)$. We will analyze this main formula using the method of steepest descent, hence we are interested in the saddle points of $F$. The saddle point equation simply reads
\begin{align*}
 F'(z)=\log z=0
\end{align*}
and is obviously satisfied only for $z=1$ only, with $F''(1)=1>0$. Denoting $z=x+iy$, the Cauchy-Riemann equations for $F$ imply 
\begin{align}
 \partial_y\Re F(z)=-\Im F'(z)=-\arg (z)>0&\Leftrightarrow y<0\label{eq:CauchyRiemann1},\\
 \partial_x\Re F(z)=\Re F'(z)=\log\abs{z}>0&\Leftrightarrow \abs z>1,\label{eq:CauchyRiemann2}.
\end{align}
Hence $\Re F$ attains its local maximum $F(1)=-1$ in $y$-direction and its minimum in $x$-direction. Define the box $Q_{\delta_n}(1)=[1-\delta_n,1+\delta_n]\times[-\delta_n,\delta_n]$ around $z=1$ of range 
\begin{align*}\delta_n=\sqrt{\frac{\log n}{ nR^{2/m}}}\le\log^{-1/4} n \to 0
\end{align*}
as $n\to\infty$ by our assumption \eqref{eq:AssR}. Note that $\tilde \gamma$ is $1/(nR^{2/m})=\mathcal O (\delta_n^2)$-close to the real axis and the vertical path $\tilde L$ is equally close to the saddle point $z=1$. The local parts of the paths are given by $L_{loc}=\tilde L\cap Q_{\delta_n}(1)$ and $\gamma_{loc}=\tilde \gamma\cap Q_{\delta_n}(1)$, as well as $L^c_{loc}$ and $\gamma^c_{loc}$ denotes the remaining part of the path (under slight abuse of notation).

Let us collect the necessary bounds for each part of the contour by applying a Taylor approximation around $z=1$. We have $(s-1)=i\Im(s)+\mathcal O (\delta_n^2)$, hence for $s\in L_{loc}$
\begin{align}\label{eq:LocalL}
F(s)=-1-\Im(s)^2/2+\mathcal O (\delta_n^3), 
\end{align}
and similarly for $t\in\gamma_{loc}$
\begin{align}\label{eq:Localg}
 F(t)=-1+(1-\Re(t))^2/2+\mathcal O (\delta^3_n),
\end{align}
since $\abs{\Im(t)}\lesssim \delta_n$. On the other hand for $s\in L_{loc}^c$ by using \eqref{eq:CauchyRiemann1} it holds
\begin{align}\label{eq:NonlocalL}
\Re F (s)<\Re F(\eta/(nR^{2/m})+i\delta_n)=-1-\delta_n^2/2+\mathcal O (\delta_n^3)
\end{align}
and for $t\in\gamma_{loc}^c$ we see from \eqref{eq:CauchyRiemann2}
\begin{align}\label{eq:Nonlocalg}
 \Re F(t)&>\Re F( 1\pm \delta_n)=-1+\delta_n^2/2+\mathcal O (\delta_n^3).
\end{align}
The nonlocal terms are negligible, e.g. we apply \eqref{eq:Localg} and \eqref{eq:NonlocalL} to obtain
\begin{align*}
 &R^{2/m}\int_{\gamma_{loc}}\int_{L_{loc}^c}\exp\left[ nmR^{2/m}\left(\Re F(s)-\Re F(t)\right)\right]\abs{\frac ts}^{m/2}\frac{\abs{\cot(\pi n R^{2/m}t)}}{\abs{t-s}}dsdt\\
 &\lesssim R^{2/m} \int_{\gamma_{loc}}\int_{L_{loc}^c}\exp\left[-\frac m2\log n +\mathcal O (\delta_n)\right]\frac{1}{\abs{s}^{m/2}\abs{\Im(s)}}dsdt\\
 &\lesssim R^{2/m} n^{-m/2}\lesssim n^{-1},
\end{align*}
where we used $\abs{\Im(s)}\gtrsim \delta_n$, $\abs{\gamma_{loc}}=\mathcal O (\delta_n)$, $t=\mathcal O (1)$, $\abs{\cot(\pi nR^{2/m} t)}\asymp 1$ and $m\ge 2$. Moreover from \eqref{eq:NonlocalL}, \eqref{eq:Nonlocalg} and $t=\mathcal O (R^{-2/m})$ it follows
\begin{align*}
 &R^{2/m}\int_{\gamma^c_{loc}}\int_{L_{loc}^c}\exp\left[ nmR^{2/m}\left(\Re F(s)-\Re F(t)\right)\right]\abs{\frac ts}^{m/2}\frac{\abs{\cot(\pi n R^{2/m}t)}}{\abs{t-s}}dsdt\\
 &\lesssim R^{2/m} \int_{\gamma^c_{loc}}\int_{L_{loc}^c}\exp\left[-m\log n +\mathcal O (\delta_n)\right]\frac{R^{-1}}{\abs{s}^{m/2}\abs{\Im(s)}}dsdt\\
 &\lesssim R^{-1} n^{-m}\lesssim n^{-1},
\end{align*}
where the last step once more follows from the assumption \eqref{eq:AssR}. Analogously we obtain from \eqref{eq:LocalL}, \eqref{eq:Nonlocalg} 
\begin{align}\label{eq:nonLocal}
 &R^{2/m}\int_{\gamma_{loc}^c}\int_{L_{loc}}\exp\left[ nmR^{2/m}\left(\Re F(s)-\Re F(t)\right)\right]\abs{\frac ts}^{m/2}\frac{\abs{\cot(\pi n R^{2/m}t)}}{\abs{t-s}}dsdt\\
  &\lesssim R^{2/m} \int_{\gamma^c_{loc}}\int_{L_{loc}}\exp\left[-\frac m2\log n +\mathcal O (\delta_n)\right]\frac{\abs t ^{m/2}}{\abs{\Re(t)-1}}dsdt\nonumber\\
  &\lesssim  \delta_n R^{-1+2/m} n^{-m/2}\lesssim n^{-1}.\nonumber
\end{align}
Locally close to $z=1$, the error term of Stirling's formula \eqref{eq:DCmu2} is $\mathcal O (n^{-1})$. Thus, it remains to control 
\begin{align}\label{eq:RemainingIntegral}
\int_{\gamma_{loc}}\int_{L_{loc}}\hspace{-4pt}\exp&\left[ nmR^{2/m}\left(F(s)-F(t)\right)\right]\left(\frac ts\right)^{\frac m2}\frac{\cot(\pi nR^{2/m} t)}{t-s}\big(R^{2/m}+\mathcal O (n^{-1})\big)dsdt\nonumber\\
 =\int_{\gamma_{loc}}\int_{L_{loc}}&\exp\left[ -\frac{nmR^{2/m}}2(\Im(s)^2+(1-\Re(t))^2)  \right]\frac{\cot(\pi nR^{2/m} t)}{t-s}\nonumber\\
& \cdot\left(R^{2/m}+\mathcal O \Big(R^{1/m}\sqrt{\frac{\log ^3 n}{ n}}\Big)\right)dsdt,
\end{align}
where we used \eqref{eq:LocalL}, \eqref{eq:Localg} and $t/s=1+\mathcal O (\delta_n)$. We parameterize $L_{loc}$ as the straight line
\begin{align*}
 s=\frac{\eta}{nR^{2/m}}+\frac{i}{\sqrt{nmR^{2/m}}}u,\quad u\in I=(-\sqrt{ m \log n},+\sqrt{ m \log n}).
\end{align*}
The vertical microscopic part 
\begin{align*}
\gamma_{vert}=[3/4+i , 3/4-i ]/(nR^{2/m})\cup([n+1/4-i, n+1/4+i ]/(nR^{2/m}))
\end{align*}
receives the same scaling, e.g. for the right part we choose
\begin{align*}
 t=R^{-2/m}+\frac1 {4nR^{2/m}}+\frac{i}{\sqrt{nmR^{2/m}}}v, \quad v\in (- \sqrt{ m/(nR^{2/m})},\sqrt{ m/(nR^{2/m})}).
\end{align*}
This part of the integral \eqref{eq:RemainingIntegral} on $\gamma_{vert}$ is visible if and only if $R$ is close to $1$. In this case, we drop the part in $t$ and the exponential function becomes $e^{-u^2/2}$.  Using $R\sim 1$ and $\abs{\cot (\pi/4+ix)}=1$ for $x\in \IR$, the integration over $\gamma_{vert}$ can then be bounded by 
\begin{align}\label{eq:VerticalContour}
 &\abs{\int_{-\sqrt{ m/(nR^{2/m})}}^{\sqrt{ m/(nR^{2/m})}}\int_I \frac{e^{-u^2/2}\cot\left(\pi/4+ivR^{1/m}\sqrt{n/m}\right)}{m(n-\eta+1/4)/R^{1/m} +i(v-u)\sqrt{nm}}R^{1/m}dudv}\\
 &\lesssim \int_{-\sqrt{ m/(nR^{2/m})}}^{\sqrt{ m/(nR^{2/m})}}\int_I \frac{e^{-u^2/2}}{\sqrt{m(n-\eta+1/4)^2/R^{2/m} +n(v-u)^2}}dudv\nonumber\\
  &\lesssim \frac{1}{n}\int_{-\infty}^\infty \frac{e^{-u^2/2}}{\sqrt{1/n +u^2}}du
 \lesssim \frac{\log n}{n},\nonumber
 \end{align}
where in the second step we shifted $u$ by $v=\mathcal O(1/\sqrt{n})$ and used $m(n-\eta+1/4)^2/R^{2/m}\gtrsim 1$. The last step follows from the asymptotics of the modified Bessel function $K_0(1/(4n))$ or more elementarily by splitting the integration into $\abs{u}\lessgtr1$. From $\delta_n\to 0$ and $1/(nR^{2/m})\to 0$ it follows that the left vertical path is not contained in $Q_{\delta_n}(1)$. 

We parameterize the remaining path of $\gamma_{loc}\setminus \gamma_{vert}$ as horizontal lines 
\begin{align*}
t=1\mp\frac{1}{\sqrt{nmR^{2/m}}}v\pm\frac{i}{nR^{2/m}},\quad v\in \tilde I\subseteq I,
\end{align*}
where $\tilde I$ is the part of $I$ such that the corresponding contour overlaps $\tilde\gamma$. The integral \eqref{eq:RemainingIntegral} becomes the
sum of two terms, which can be estimated as
\begin{align}\label{eq:ContourIntegrability}
&\Bigg\vert\int_{\tilde I}\int_I e^{-\frac{u^2+v^2}{2}}\frac{\cot(\pi nR^{\frac 2 m}\pm i\pi \mp \pi \sqrt{nR^{\frac 2 m}/m}v) }{ \mp v+(nR^{\frac 2 m}-\eta)\sqrt{\frac{m}{nR^{\frac 2 m}}} +i\left(\pm \sqrt{\frac{m}{nR^{\frac 2 m}}}-u \right) }\nonumber\\
&\qquad\frac{iR^{\frac 1 m}+\mathcal O (\sqrt{\frac{\log ^3 n}{n}})}{\mp\sqrt{nm}}dudv\Bigg\vert\\
&\le\frac{1}{\sqrt{nm}}\int_{\tilde I}\int_I e^{-\frac{u^2+v^2}{2}}\frac{R^{\frac 1 m}+\mathcal O\Big(\sqrt{\frac{\log ^3 n}{n}}\Big)}{ \sqrt{\left(v\mp(nR^{\frac 2 m}-\eta)\sqrt{\frac{m}{nR^{\frac 2 m}}}\right)^2 +\left(u\mp \sqrt{\frac{m}{nR^{\frac 2 m}}}\right)^2}  }dudv\nonumber\\
%&=\frac{1}{\sqrt{nm}}\int\int e^{-\frac{u^2+v^2}{2}}\frac{\left(1+\mathcal O(\frac{\log n}{\sqrt n R^{\frac 1 m}})  \right)\left(R^{\frac 1 m}+\mathcal O (\sqrt{\frac{\log ^3 n}{n}})\right)}{ \sqrt{v^2 +u^2}  }dudv\nonumber\\
&\le\frac{1}{\sqrt{nm}}\int_{\IR^2}e^{-\abs{z+\mathcal O(\sqrt{m/(nR^{2/m})})}^2/2}\frac{R^{\frac 1 m}+\mathcal O\Big(\sqrt{\frac{\log ^3 n}{n}}\Big)}{\abs{z}} d\lebesgue(z)\nonumber
\end{align}
where we shifted $z=v+ui\in I+i\tilde I$ by the corresponding values from the first denominator, which are of order $\mathcal O(\sqrt{m/(nR^{2/m})})=\mathcal O (\log^{-\frac 3 4} n)$, and extended the area of integration. A simple estimate on the errors in the exponential function yields
\begin{align*}
&\frac{1}{\sqrt{nm}}\int_{\IR^2}e^{-\abs{z}^2/2} \left( 1+(\abs z +1)\mathcal O \Big( \sqrt{\frac{m}{nR^{\frac 2 m}}}\Big)\right) 
\frac{R^{\frac 1 m}+\mathcal O\Big(\sqrt{\frac{\log ^3 n}{n}}\Big)}{\abs{z}}d\lebesgue(z)\nonumber\\
&\le\frac{1}{\sqrt{nm}}\int_{\IR^2}e^{-\abs z^2/2}\frac{1+\mathcal O\Big(\sqrt{\frac{\log ^3 n}{n}}\Big)}{\abs{z}} d\lebesgue(z)\sim\sqrt{\frac{2\pi^3}{nm}}\nonumber.
\end{align*}
Recalling the correct prefactor  $c=-1/4\pi$ from \eqref{eq:DCmu2}, we conclude
\begin{align*}
\abs{\mu_{\infty}^m(B_R)-\bar\mu_n^m(B_R)}\leq \sqrt{\frac{\pi}{2nm}}+o (n^{-1/2}).
\end{align*}
In order to obtain the lower bound of the claim, it suffices to consider $R=1$. Moreover we will only study the sum of \eqref{eq:ContourIntegrability} keeping the phase factor of the integrand, since all the other parts of the double contour integral are proven to be of strictly lower order than $o (n^{-1/2})$. We have the approximation
\begin{align*}
&\mu_{\infty}^m(B_1)-\bar\mu_n^m(B_1)\\
 &=\frac{1}{4\pi}\Bigg(\int_{-\sqrt m/(4\sqrt n)}^{\sqrt{m \log n}}\int_{-\sqrt{m \log n}}^{+\sqrt{m \log n}}e^{-\frac{u^2+v^2}{2}}\frac{\cot(i\pi - \pi \sqrt{n/m}v) }{v+\frac{\sqrt m}{2\sqrt n} +i\left(-\sqrt{\frac mn}+u\right) }\frac{i+o(1)}{\sqrt{nm}}dudv\\
 &+ \int_{-\sqrt{m \log n}}^{+\sqrt m/(4\sqrt n)}\int_{-\sqrt{m \log n}}^{+\sqrt{m \log n}}e^{-\frac{u^2+v^2}{2}}\frac{\cot(-i\pi + \pi \sqrt{n/m}v) }{v-\frac{\sqrt m}{2\sqrt n} +i\left(- \sqrt{\frac mn}-u\right) }\frac{i+o(1)}{\sqrt{nm}}dudv\\
&=\frac{1}{4\pi\sqrt{nm}}\int_{\sqrt m/ (4\sqrt n)}^{\sqrt{m \log n}}\int_{-\sqrt{m \log n}}^{+\sqrt{m \log n}}e^{-\frac{u^2+v^2}{2}}\Big(1 +\mathcal O\Big(\sqrt{\tfrac{\log n}{n}} \Big) \Big)\frac{-i}{v+ui}\\
&\qquad \cdot(\tan(i\pi-\pi\sqrt{n/m}v)+\tan(i\pi+\pi\sqrt{n/m}v))
dudv+o (n^{-1/2}).
\end{align*}
In the last step we first proceeded similarly to \eqref{eq:ContourIntegrability}, i.e. we shifted $v$ by a negligible value $\pm\sqrt m/(2\sqrt n)$ and $u$ by $\pm \sqrt{m/n}$,  and then inverted $v\to -v$ in the second integral in order to merge both and put $\cot(\pi/2-z)=-\tan(-z)=\tan(z)$. Note that $\sup_{x\in\IR}\abs{\tan(i\pi-x)+\tan(i\pi+x)-2i}=:\epsilon<0.01$ and that the left hand side of the previous equation is real, hence
\begin{align*}
 &\mu_{\infty}^m(B_1)-\bar\mu_n^m(B_1)\\
 &\ge \frac{1-\epsilon}{2\pi\sqrt{nm}}\int_{\sqrt m/ (4\sqrt n)}^{\sqrt{m \log n}}\int_{-\sqrt{m \log n}}^{+\sqrt{m \log n}}e^{-\frac{u^2+v^2}{2}}\frac{v}{u^2+v^2}dudv+o(n^{-1/2})\\
 &\sim\frac{1-\epsilon}{\sqrt{2\pi m n}}.
\end{align*}
The same upper bound holds with $1+\epsilon$ instead. For a better control of the constant one may vary the distance of $\gamma$ to the real axis from the start. The above asymptotic yields the first claim and coincides with $m=1$ from Lemma \ref{lem:GinibreRate}. 

If we avoid the edge by some distance $\abs{1-R^{1/m}}\gtrsim\sqrt{\log n/ n}>0$, then $\delta_n<\abs{1-1/R^{2/m}}$. Hence $\gamma_{vert}$ is not a part of $\gamma_{loc}$ and \eqref{eq:VerticalContour} drops out. This is the case for $R<1-\tfrac m2 \sqrt{\log n/ n}$
, for which we have $I=\tilde I$. As before, we shift $u$ and $v$ by $\mathcal O (1/\sqrt{nR^{2/m}})$, now we denote the shifted intervals by $I_1, I_2$ and obtain
\begin{align*}
&\int_{ I}\int_I e^{-\frac{u^2+v^2}{2}}\Bigg(\frac{-\cot(\pi nR^{\frac 2 m}+ i\pi - \pi \sqrt{nR^{\frac 2 m}/m}v) }{ - v+(nR^{\frac 2 m}-\eta)\sqrt{m/(nR^{\frac 2 m})} +i\big( \sqrt{m/(nR^{\frac 2 m})}-u \big) }\\
&+\frac{\cot(\pi nR^{\frac 2 m}- i\pi + \pi \sqrt{nR^{\frac 2 m}/m}v) }{ v+(nR^{\frac 2 m}-\eta)\sqrt{m/(nR^{\frac 2 m})} -i\big(\sqrt{m/(nR^{\frac 2 m})}+u \big) }\Bigg)\frac{iR^{\frac 1 m}+\mathcal O (\sqrt{\frac{\log ^3 n}{n}})}{\sqrt{nm}}dudv\\
&=\int_{I_1}\int_{I_2} e^{-\frac{u^2+v^2}{2}}\Big(1 +(1+v+ui)\mathcal O\Big(\tfrac{1}{\sqrt{nR^{\frac2m}} }\Big)\Big) \frac{iR^{\frac 1 m}+\mathcal O (\sqrt{\frac{\log ^3 n}{n}})}{\sqrt{nm}(v+ui)}\\
&\qquad \cdot\left(\tan(i\pi - \pi \sqrt{nR^{\frac 2 m}/m}v) -\tan(- i\pi + \pi \sqrt{nR^{\frac 2 m}/m}v)\right)dudv\\
&=\int_{ I}\int_I e^{-\frac{u^2+v^2}{2}}\frac{iR^{\frac 1 m}+\mathcal O (\sqrt{\frac{\log ^3 n}{n}})}{\sqrt{nm}(v+ui)}\\
&\qquad \cdot\left(\tan(i\pi - \pi \sqrt{nR^{\frac 2 m}/m}v) -\tan(- i\pi + \pi \sqrt{nR^{\frac 2 m}/m}v)\right)dudv.
\end{align*}
As we have already showed and used in \eqref{eq:ContourIntegrability}, the error terms of the shifted variables in the exponential function are absorbed in the error term $\mathcal O (\sqrt{\log ^3 n /n}) $. Furthermore, the deviation between the integrals over $I$ instead of $I_1, I_2$ are obviously negligible as well. From the last line it follows that the horizontal contour integrals over $I$ cancel due to symmetry, thus only the error term remains. Collecting this error, together with the non-local terms, we obtain
\begin{align*}
\mu_{\infty}^m(B_R)-\bar\mu_n^m(B_R)\lesssim\frac{1}{\sqrt n}\int_{-\infty}^{\infty}\int_{-\infty}^{\infty} e^{-\frac{u^2+v^2}{2}}\frac{\sqrt{\log^3 n}}{\sqrt n\abs{v+ui}}dudv+\mathcal O \Big(\frac 1 n\Big)\lesssim\frac{\log^{3/2} n}{n}.
\end{align*}

Lastly we turn to the statement about the exponential decay for $R>1+\sqrt{\log n/n}$. As before it is sufficient to consider $R\leq2$, because of $\supp(\mu_\infty^m)=B_1$. The position of the minimum in $x$-direction in \eqref{eq:CauchyRiemann2} and $\Re( t)\leq (n+1)/(nR^{2/m})$ for $t\in\tilde\gamma$ yield
\begin{align}\label{eq:Nonlocalg2}
 \Re F(t)&=\Re F(\Re (t))+\mathcal O(1/n)\nonumber\\
 &\ge F\left(\frac{n+1}{nR^{2/m}}\right)+\mathcal O(1/n)\\
 & =-R^{-2/m}\left(\log (R^{2/m})+1\right)+\mathcal O(1/n)\nonumber
\end{align}
for $n$ sufficiently large. Note that in this case of big $R$, the local area $Q_{\delta_n}(1)$ does intersect $\tilde\gamma$, hence $\gamma_{loc}=\emptyset$. We estimate \eqref{eq:DCmu2} similarly to \eqref{eq:nonLocal}, i.e. apply \eqref{eq:LocalL}, \eqref{eq:NonlocalL}, \eqref{eq:Nonlocalg2} to obtain 
\begin{align*}
 &\int_{\tilde\gamma}\int_{\tilde L}\exp\left[ nmR^{2/m}\left(\Re F(s)-\Re F(t)\right)\right]\abs{\frac ts}^{m/2}\frac{\abs{\cot (\pi n R^{2/m}t)}}{\abs{t-s}}dsdt\\
 &\lesssim \int_{\tilde\gamma}\int_{\tilde L}\exp\left[-nm(R^{2/m}-1 -\log(R^{2/m})) +\mathcal O (1)\right]\frac{1}{\abs{s}^{m/2}\abs{\Im(s)}}dsdt.
 \end{align*}
In particular, we used the fact that $\abs{t-s}$ behaves like $\abs{\Im (s)}$ for $s\in L_{loc}^c$ and is bounded from below by $\Re(s-t)\gtrsim \sqrt{\log n / n}\gtrsim \abs{\Im (s)}$ for $s\in L_{loc}$, since $\Re(t)$ is small. The remaining integral over $t\in\tilde\gamma$ is finite, since $\tilde\gamma$ has finite length and the integral over $s\in\tilde L$ is finite, hence we may use Bernoulli's inequality to conclude
 \begin{align*}
  \abs{\mu_{\infty}^m(B_R)-\bar\mu_n^m(B_R)}\lesssim \exp\left[-2n(R-1-\log R)\right]\le\exp\left[-n(R-1)^2/3\right], 
\end{align*}
where the last inequality holds for $1<R<2$ by a Taylor approximation of the logarithm. 
Ultimately all claims are proven.
\end{proof}

It seems to be an artifact of the method of proof that we do not obtain an exponential rate of convergence inside the bulk in the case of products of Gaussian random matrices. Only the rate $\mathcal O (1/n)$ seems to be achievable due to the discrete nature of the residue calculus, cf. \eqref{eq:DCmu2} above. 

 The proof of the rate of convergence for the non-averaged ESD makes use of the previous result, the determinantal structure of $\mu_n^m$ and Bernstein's inequality.

 \begin{proof}[Proof of Theorem \ref{thm:ProductRate}]
 Without loss of generality, we consider $ 0<R<1$, since for $R\ge1$, we have $|\mu_n^m(B_R)-\mu_\infty(B_R)|\le|\mu_n^m(B_1)-\mu_\infty(B_1)|$. %By a union bound argument it is sufficient to consider a fixed radius: 
Suppose we had shown already that for any $Q>0$ there exists a $c>0$ such that
\begin{align}\label{eq:fixedR}
 \IP\Big(\abs {\mu_n^m(B_R )-\mu_\infty^m(B_R )}>t_n/5\Big)\le n^{-Q-\lceil m/4\rceil}
\end{align}
for $t_n=c\sqrt{\log n/ n}$ and fixed $0<R<1$. Set $p=\lceil m/4\rceil$ and consider the equidistant points $r_k=k/n^p$ for $k=1,\dots,n$. We have
\begin{align*}
\abs{\mu_n^m(B_{R} )-\mu_\infty^m( B_{R} )}\le \mu_n^m(B_{r_{k+1}} \hspace{-2pt}\setminus\hspace{-2pt} B_{r_{k}} )+\mu_\infty^m(B_{r_{k+1}} \hspace{-2pt}\setminus\hspace{-2pt} B_{r_{k}} )+\abs{\mu_n^m(B_{r_{k}})-\mu_\infty^m (B_{r_{k}} )}
\\ \le \abs{\mu_n^m(B_{r_{k+1}})-\mu_\infty^m (B_{r_{k+1}} )}+2\mu_\infty^m(B_{r_{k+1}} \hspace{-2pt}\setminus\hspace{-2pt} B_{r_{k}} )+2\abs{\mu_n^m(B_{r_{k}})-\mu_\infty^m (B_{r_{k}} )}
\end{align*}
where we chose $k=\lfloor Rn^p\rfloor$.
Taking the supremum over $R$ is equivalent to taking the maximum in $k$, hence the union bound implies
\begin{align*}
 \IP\Big(\sup_{R>0}\abs{\mu_n^m(B_{R} )-\mu_\infty^m( B_{R} )}> t_n\Big)
\le&
 \sum_{k=1}^{n^p}\Big(\IP\big(\abs{\mu_n^m(B_{r_{k+1}})-\mu_\infty^m (B_{r_{k+1}} )}>  t_n/5\big)\\
  &+ \IP\big(\mu_\infty^m(B_{r_{k+1}} \setminus B_{r_{k}} )>  t_n/5\big)\\
  &+ \IP\big(\abs{\mu_n^m(B_{r_{k}})-\mu_\infty^m (B_{r_{k}} )}>  t_n/5\big)\Big).
\end{align*}
The first and last term are covered by \eqref{eq:fixedR} and the second term vanishes because of 
\begin{align*}
 \mu_\infty^m(B_{r_{k+1}} \setminus B_{r_{k}})=\frac{(k+1)^{2/m}-k^{2/m}}{n^{2p/m}}\le n^{-p(1\wedge 2/m)}\le n^{-1/2}.
\end{align*}
Thus, we have shown that for all $Q>0$ there exists a constant $c>0$ such that $\sup_{R>0}\abs{\mu_n^m(B_{R} )-\mu_\infty^m( B_{R} )}\le c \sqrt{\log n/ n}$ holds  with probability at least $1- n^{-Q}$.
 
It remains to show \eqref{eq:fixedR} for which we follow the ideas of \cite[Proposition 2.2]{MM14rate} for the case $m=1$. Fix $R<1$ and let $\xi_k\sim Ber(\eta_k)\in\{0,1\}$ be independent Bernoulli variables with parameters $\eta_k\in[0,1]$, which shall be defined in the following. According to \cite[Corollary 4.2.24]{AGZ}, the determinantal point process 
\begin{align*}
\#\{\lambda_j(\sqrt{n^m}\mathbf X)\in B_{\sqrt{n^m}R} \}=n\mu_n^m(B_R )\stackrel{\mathcal D}{=}\sum_{k=1}^n\xi_k
\end{align*}
has the same distribution as the sum of Bernoulli variables, where the parameters $\eta_k$ are given by the eigenvalues of the trace class operator %(NOTE: the following two lines are actually not important right now)
\begin{align*}
 \mathcal K_R f(z)=\int_{B_{\sqrt{n^m}R} } \sqrt{G^{m,0}_{0,m}\Big(\genfrac{}{}{0pt}{}{-}{0}\Big\vert \abs z ^2\Big)G^{m,0}_{0,m}\Big(\genfrac{}{}{0pt}{}{-}{0}\Big\vert \abs w ^2\Big)} \sum_{k=0}^{n-1} \frac{(z\bar w)^k}{\pi k!^m} f(w) d\lebesgue(w)
\end{align*}
for $f\in L^2(B_{\sqrt{n^m}R} )$. Due to rotational symmetry, like we argued for the orthogonality of the monomials, $\mathcal K_R$ has eigenfunctions $\phi_k(w)=\sqrt{G^{m,0}_{0,m}\Big(\genfrac{}{}{0pt}{}{-}{0}\Big\vert \abs w ^2\Big)}w^k$ with eigenvalues
\begin{align*}
 \eta_k=\int_{B_{\sqrt{n^m}R} } G^{m,0}_{0,m}\Big(\genfrac{}{}{0pt}{}{-}{0}\Big\vert \abs w ^2\Big) \frac{\abs w ^{2k}}{\pi k!^m} d\lebesgue(w)\le 1.
\end{align*}
It follows from Theorem \ref{thm:ProductMeanRate} that 
\begin{align*}
 \IP\Big(\abs {\mu_n^m(B_R )-\mu_\infty^m(B_R )}\ge \frac {t}{\sqrt n}\Big)&\le \IP\Big(\abs {\mu_n^m(B_R )-\bar\mu_n^m(B_R )}\ge \frac{t- c}{\sqrt n}\Big)\\
 &=\IP\Big(\Big\vert\sum_{k=1}^n\xi_k-\IE\Big(\sum_{k=1}^n\xi_k\Big) \Big\vert\ge (t-c)\sqrt n\Big).
\end{align*}
Applying Bernstein's inequality, see \cite[Equation (2.10)]{BLMconc} yields
\begin{align}\label{eq:Bernstein}
 \IP\Big(\abs {\mu_n^m(B_R )-\mu_\infty^m(B_R )}>\frac{t}{\sqrt{n}}\Big)&\le 2\exp\Big( -\frac{n(t-c)^2}{2\sum_{k=1}^n\IE\abs{\xi_k}^2+\tfrac 2 3(t-c)\sqrt n}\Big)\nonumber\\
& \le e^{-t^2/3}
\end{align}
for $t$ sufficiently large, since $\IE\abs{\xi_k}^2\le\eta_k\le 1$. In particular for $t=c\sqrt{\log n}/5$ with $c=5\sqrt{3(Q+\lceil m/4\rceil)}$, we obtain \eqref{eq:fixedR} as claimed. 
\end{proof}

 Note that the logarithmic factor in the rate of convergence is expected to be non-optimal and might be removed by controlling the variance term \eqref{eq:Bernstein}, but this will not be pursued here.

\section{Matrices with independent entries}\label{sec:iidMatrices}
Let us begin with some necessary notations. We define the \emph{linearization matrix} as the $mn\times mn$-block matrix
\begin{align*}\mathbf W:=\frac 1{\sqrt n}
\begin{pmatrix}
0 & X^{(1)} & 0 & \cdots & 0\\
\vdots & 0& X^{(2)} & \ddots &\vdots\\
\vdots& & \ddots&\ddots& 0\\
0&\cdots&\cdots&0& X^{(m-1)}\\
X^{(m)} &0  & \cdots &\cdots& 0\\
\end{pmatrix}
\end{align*}
and note that $\mathbf W^m$ is a block diagonal matrix with cyclic products of $X^{(1)},\dots,X^{(m)}$. Consequently, its eigenvalues consist of the eigenvalues $\lambda_j(\mathbf X)$ of $\mathbf X$ with multiplicity $m$. Furthermore define its shifted Hermitization
\begin{align}\label{eq:V}
 \mathbf V(z):=\begin{pmatrix}
  0 & \mathbf W -z\\
  (\mathbf W -z)^* & 0
 \end{pmatrix}
\end{align}
for $z\in\IC$. The eigenvalues of $\mathbf V (z)$ are given by $\pm s_j(\mathbf W -z)$, where $s_{\max}=s_1\geq\dots\geq s_{mn}=s_{\min}$ are the singular values of $\mathbf W -z$. Its ESD shall be denoted as $\tilde\nu^z_n$.

Similar to the role of the Stieltjes transform in the theory of Hermitian random matrices, the weak topology of measures $\mu$ on $\IC$ can be expressed in terms of the so called logarithmic potential $U$, which is the solution of the distributional Poisson equation. More precisely for every finite Radon measure $\mu$ on $\IC$ the \emph{logarithmic potential} defined by
\begin{align}\label{eq:logPot}
 U_\mu(z):=-\int_{\IC}\log\abs{t-z}d\mu(t)=(-\log\abs \cdot\ast\mu)(z)
\end{align}
and it satisfies
\begin{align}\label{eq:distrPoisson}
    \Delta U_\mu=-2\pi \mu 
\end{align}
in the sense of distributions. Let $U_n$ denote the logarithmic potential of the ESD of $\mathbf W$. The advantage of the logarithmic potentials $U_n$ of $\mu_n$ in non-Hermitian random matrix theory is the following identity known as \emph{Girko's Hermitization trick}
\begin{align}
U_n(z)&=-\frac{1}{nm}\sum_{j=1}^{nm}\log\abs{\lambda_j(\mathbf W)-z}=-\frac{1}{2nm}\log\abs{\det \mathbf V(z)} =-\int_{-\infty}^\infty\log \abs xd\tilde\nu_n^z(x).\label{eq:GirkoHerm}
\end{align}
 Under the above-mentioned conditions on the matrix entries, the logarithmic potential $U_n$ concentrates around the logarithmic potential $U_\infty$ of the Circular Law given by
\begin{align*}
 U_\infty(z)=\begin{cases}
              -\log\abs z &\text{, if }\abs z >1\\
              \tfrac 1 2 (1-\abs z ^2) &\text{, if }\abs z \le 1
             \end{cases}.
\end{align*}

\begin{prop}[\cite{GNT17Local}]\label{prop:ConcLogPot}
If $\mathbf X$ obeys (C), then for every $\tau,Q>0$ there exists a constant $c>0$ such that 
\begin{align}
\IP\left(\abs{U_n(z)-U_\infty(z)}\leq c\frac{\log^4 n}{n}\right)\geq 1-n^{-Q}
\end{align}
holds uniformly in $\{z\in B_{1+\tau^{-1}}: \abs{1-\abs z }\ge \tau\}$.
\end{prop}

We remark that the restriction to the bulk in Theorem \ref{thm:rateofconv} comes directly from the restriction in the previous Proposition \ref{prop:ConcLogPot}.

Since the statement of Proposition \ref{prop:ConcLogPot} is not explicitly worked out in \cite{GNT17Local}, we will derive it now, based on the results proved in this paper.
We will directly follow the approach of \cite{GNT17Local}, making use of Girko's Hermitization trick to convert the non-Hermitian problem into a Hermitian one, apply the local Stieltjes transform estimate from \cite{GNT17Local} and the smoothing inequality from \cite{GT03RateSemicircular}. Let
\begin{align*}
 m_n(z,\cdot ):\IC\setminus\IR\to\IC,\qquad w\mapsto\int_{\IR}\frac{1}{w-t}d\tilde\nu^z_n(t)
\end{align*}
be the Stieltjes transform of $\tilde\nu_n^z$, which converges a.s. to the solution of 
\begin{align}\label{eq:StieltjesEquation}
s(z,w)=-\frac{w+s(z,w)}{(w+s(z,w))^2-\abs z ^2}
\end{align}
see for instance \cite{GT10Circular}. It is known that $s(z,\cdot )$ is the Stieltjes transform of a measure $\tilde \nu ^z$, which appears to be the weak limit of $\tilde\nu^z_n$ as $n\to\infty$, cf. \cite{GNT17Local}. Moreover, $\tilde \nu ^z$ has a symmetric bounded density $\rho^z$ (the bound holds uniformly in $z$) and has compact support. We refer to \cite[\S 4]{BYY14Local} for more information on $s$ and $\rho^z$.
Note that $s$ will be unbounded from below for $w$ close to the origin and $z$ close to the edge of the support of $\mu^m_\infty$, which is the reason for the bulk constraint of Proposition \ref{prop:ConcLogPot}.

\begin{proof}[Proof of Proposition \ref{prop:ConcLogPot}]
Fix some arbitrary $Q,\tau>0$ and $z\in B_{1+\tau^{-1}}$ satisfying $\abs{1-\abs z }\ge \tau$. As is explained in Girko's Hermitization trick \eqref{eq:GirkoHerm}, we may write
\begin{align*}
 \abs{U_n(z)-U_\infty(z) }=\abs{\int_{\IR} \log \abs x d(\tilde\nu^z_n-\tilde\nu^z)(x)	}.
\end{align*}
Due to the singularities of $\log\abs{x}$ at $x=0$ and $x=\infty$, it is not enough to simply control the convergence of $\tilde\nu^z_n\to\tilde\nu^z$, but it is also necessary to estimate the extremal singular values. The rate of convergence of $\tilde\nu^z_n$ to $\tilde\nu^z$ shall be measured in Kolmogorov distance $$ d_n^*(z)=\sup_{x\in\IR}\abs{(\tilde\nu_n^z-\tilde\nu^z)(-\infty,x]}.$$ Introduce the events
\begin{align*}
 \Omega_0:=\{s_{\min }\geq n^{-B}  \},\quad \Omega_1:=\{s_{\max }\leq n^{B'}  \},\quad \Omega_2:=\{  d_n^*(z)\leq c \log^3n/n\}
\end{align*}
for some constants $B,B',c>0$ yet to be chosen. From \cite[Theorem 31]{OS11} it follows that there exists a constant $B>0$ such that $\IP(\Omega_0^c)\lesssim n^{-Q}$. Moreover for any $Q>0$ it holds
\begin{align}\label{eq:s_max}
 \IP(s_{\max }\geq n^{(Q+1)/2})&\leq \frac {\IE\norm{(\mathbf W-z)}^2} {n^{Q+1}}\\
 &\leq \frac 1 {n^{Q+1}}\sum_{ij}^{nm}\IE\abs{\mathbf W_{ij}}^2+\frac {\abs z^2} {n^{Q}}\le\big(m +\tau^{-1}\big)n^{-Q},\nonumber
\end{align}
where the operator norm $\norm\cdot$ has been estimated by the Hilbert Schmidt norm. Thus, there exists a constant $B'>0$ with $\IP(\Omega_1^c)\lesssim n^{-Q}$. Since $\tilde\nu^z$ has a bounded density, we get
\begin{align*}
 \abs{\int_{-n^{-B}}^{n^{-B}}\log \abs x d\tilde\nu^z(x)}\lesssim \frac{\log n}{ n^{B}}
\end{align*}
and furthermore on $\Omega_2$ it holds that
\begin{align*}
 \abs{\int_{n^{-B}\leq \abs x\leq n^{B'}}\log\abs xd(\tilde\nu_n^z-\tilde\nu ^z )(x)}\lesssim   d^*_n(z)\log n\lesssim \frac{\log ^4n}{n}.
\end{align*}
Hence, the claimed concentration of $U_n$ holds on $\Omega_0\cap\Omega_1\cap\Omega_2$, implying
\begin{align*}
 \IP\left( \abs{U_n(z)-U_\infty(z) } \geq c \frac{\log ^4 n}{n}\right)\leq \IP(\Omega_0^c)+\IP(\Omega_1^c)+\IP(\Omega_2^c).
\end{align*}
It remains to check $\IP(\Omega_2^c)\leq n^{-Q}$, which has been done explicitly in \cite[(4.14)-(4.16)]{GNT17Local}, using the Smoothing Inequality \cite[Corollary B.3]{GNT17Local} (originally obtained in \cite{GT03RateSemicircular}) and the Local Law for $\mathbf V(z)$ in terms of its Stieltjes transform.
\end{proof}

The core of the proof of the local law for products of non-Hermitian matrices in \cite{GNT17Local} is the following identity. First, for any function $f\in\mathcal C ^2_c(\IC)$ define $\tilde f$ by $\tilde f (z)=f(z^m)$ and note that $\int fd\mu_\infty ^m=\int \tilde fd\mu_\infty ^1$, which follows from the definition of $\mu_\infty^m$ in \eqref{eq:LimitProduct}. Using the distributional Poisson equation \eqref{eq:distrPoisson} as usual in non-Hermitian random matrix theory and the representation of the eigenvalues of $\mathbf W$, we get
\begin{align}\label{eq:LocalLawIdentity}
 \int f d(\mu_n^m-\mu^m_\infty)=\frac 1 {nm}\sum_{j=1}^{nm}\tilde f (\lambda_j(\mathbf W))-\int\tilde fd\mu_\infty ^1=-\frac{1}{2\pi}\int\Delta \tilde f \left(U_n- U_\infty\right)d\lebesgue.
\end{align}

In Local Laws \cite{nemish,GNT17Local}, locally shrinking functions $f_{z_0}$ have been considered and for fixed global $f$ Gaussian fluctuation has been proven in \cite{KOV18}. As was done in \cite{GJ18Rate}, we would like to uniformly approximate all indicator functions $\eins_{B_R(z_0)}$ by smooth functions, replace the right hand side of \eqref{eq:LocalLawIdentity} by a discrete random sum and use the pointwise estimate from Proposition \ref{prop:ConcLogPot}. In contrast to \cite{GJ18Rate}, we cannot use the smoothing inequality \cite[Theorem 2.1]{GJ18Rate}, since it would rely on the distance of $U_\infty$ to the logarithmic potential of the matrix $\mathbf X$, while instead we can control the logarithmic potential $U_n$ of the linearization matrix $\mathbf W$ only. Therefore we will use a direct approach. In this case, we can replace the Monte Carlo approximation of the integral \eqref{eq:LocalLawIdentity} by a random grid approximation. This also yields a more precise rate of convergence in Theorem \ref{thm:rateofconv}. Here, and in the sequel we will identify $\IR^2\simeq\IC$ via $(u,v)\mapsto u+iv$ in order to simplify notation.

\begin{lem}\label{lem:gridapprox}
Let $\alpha,\beta>0$ and $S$ be a random variable uniformly distributed on $[0,1]^2$. Fix complex numbers $\lambda_1,\dots,\lambda_n\in\IC$ with corresponding logarithmic potential $U$ of their empirical distribution $\mu_n=\tfrac 1 n\sum_{j=1}^n\delta_{\lambda_j}$. Furthermore, define the random grid $A=2\beta n^{-\alpha/2}(\IZ^2+S)\cap [-\beta,\beta]^2$ enumerated by $z_1,\dots,z_{\lceil n^{\alpha}\rceil}$. For any function $f\in\mathcal C ^3_c(\IC)$ with $\supp f\subseteq (-\beta,\beta)^2$ it holds
\begin{align}\label{eq:gridapprox}
 \frac 1 n \sum_{j=1}^n f(\lambda_j) - \frac{-2\beta^2}{n^{\alpha}\pi}\sum_{i=1}^{\lceil n^{\alpha}\rceil}\Delta f(z_i)U(z_i)=\mathcal O\big(\big(\norm{\nabla \Delta f }_\infty+\norm{\Delta f }_\infty \log^2 n \big)n^{-\alpha/2}\big)
\end{align}
with overwhelming probability. Moreover, if $S$ is chosen independently of the random matrix $\mathbf W$, then \eqref{eq:gridapprox} for $U=U_n$ and
\begin{align*}
\sup_{i}\abs{U_n(z_i)}=\mathcal O (\log^2n)
\end{align*}
hold on an event $\Omega_\ast$ of probability $1-\mathcal O(n^{-\log n})$, which does not depend on $f$.
\end{lem}

%In the proof of Theorem \ref{thm:rateofconv} we will apply this Lemma to the eigenvalues $\lambda_j$ of $\mathbf W$.
Note that \eqref{eq:gridapprox} holds uniformly in $f\in\mathcal C ^3_c(\IC)$, hence we could choose a function depending on the positions of $\lambda_j$. In order to make the statement more intuitive, suppose we replace the logarithmic potential $U$ by a more regular function $U\in\mathcal C^1$. Then \eqref{eq:gridapprox} is nothing but Riemann approximation of the integral
\begin{align}\label{eq:RiemannApprox}
 \int \Delta f(z)U(z)d\lebesgue (z)-\frac{(2\beta)^2}{n^\alpha}\sum_{i=1}^{\lceil n^{\alpha}\rceil}\Delta f(z_i)U(z_i) \lesssim\left(\norm{\nabla \Delta f }_\infty+\norm{\Delta f }_\infty\right) n^{-\alpha/2}.
\end{align}
This follows directly from the mean value theorem, very similarly to what we will do in \eqref{eq:MVT} below.

In the Monte Carlo approximation used in \cite{TV15uni} and \cite{KOV18}, the random points $z_i$ are not ordered in a grid but drawn independently, thus variance bounds are of importance for improving bounds as \eqref{eq:gridapprox}. By using reference points or eigenvalue rigidity, the error estimates in \cite{TV15uni} and \cite{KOV18} are stronger by a factor of $1/n$ for the same number of points $z_i$. On the other hand, in order to control the singularities of $U_n$, one has to handle many random effects of all $z_i$, whereas in \eqref{eq:gridapprox} only a single random shift affects all points $z_i$. Heuristically speaking, this leads to a higher probability than in previous approaches, so that the weaker error bound is negligible. 

\begin{proof}[Proof of Lemma \ref{lem:gridapprox}]
Using the definition \eqref{eq:logPot} via \eqref{eq:distrPoisson}, in other words integration by parts, we find
\begin{align*}
  \frac 1 n \sum_{j=1}^n f(\lambda_j)=\int f d\mu_n=-\frac1{2\pi}\int\Delta f(z)U(z)d\lebesgue(z).
\end{align*}
It suffices to show that with probability at least $1-n^{-\log n-1}$ we have
\begin{align}\label{eq:fixedlambda}
 \int\Delta f(z)\log\abs{\lambda-z}d\lebesgue(z)-\frac{4\beta^ 2}{n^{\alpha}}\sum_{i=1}^{\lceil n^{\alpha}\rceil}\Delta f(z_i)\log\abs{\lambda-z_i}\nonumber\\
=\mathcal O\big(\big(\norm{\nabla \Delta f }_\infty+\norm{\Delta f }_\infty \log^2 n \big)n^{-\alpha/2}\big)
\end{align}
for fixed $\lambda\in\IC$, since the claim then follows from summation over $j$ and the union bound.

In the following, we will show that the event, where \eqref{eq:fixedlambda} holds, will be $\tilde\Omega_\ast=\bigcap_{z_i\in A}\{\abs{z_i-\lambda}> 2 \beta n^{-(\log n+1+\alpha)/2}\}$ which fails if $\lambda$ is too close to the random grid $A$. More precisely let $z^\ast\in 2\beta n^{-\alpha/2}\IZ^2\cap[-\beta,\beta]^2$ be the corner of a box such that $\lambda\in z^\ast+[0,2\beta n^{-\alpha/2}]^2$, then
\begin{align*}
 \IP(\tilde\Omega_\ast^c)&=\IP\left(\exists i=1,\dots,{\lceil n^{\alpha}\rceil}:\abs{z_i-\lambda}\le2\beta n^{-(\log n+1+\alpha)/2}\right)\\
 &=\IP\left(\text{dist}\Big( S,\frac{\lambda-z^\ast}{2\beta n^{-\alpha/2}}\Big) \le n^{(-\log n-1)/2}\right)=\mathcal O(n^{-\log n-1}), 
\end{align*}
where the distance in $[0,1]^2$ is measured according to the metric of the quotient space $\mathbb T ^2$. The situation is illustrated in Figure \ref{fig:grid}.

\begin{figure}[h]
 \begin{tikzpicture}
\draw[gray,line width=.6pt,<->] (3.5,2.7) -- (3.7,3.3);
\draw[gray,line width=.6pt,->] (0.6,0) -- (.7,.3);
\draw[gray,line width=.6pt,<-] (0.5,1.2) -- (.6,1.5);
\foreach \y in {0,1.5,3,4.5}{
\foreach \x in {0,1.5,3,4.5}{
  \begin{scope}[shift={(.5,1.2)}]
  \draw[gray, line width=0.5pt, dotted] (-.7,\y) -- (4.9,\y);
\draw[gray, line width=0.5pt, dotted] (\x,-1.4) -- (\x,4.6);
\filldraw[gray] (\x,\y) circle (.7pt);
  \end{scope}
\draw[line width=0.5pt] (-.5,\y) -- (5,\y);
\draw[line width=0.5pt] (\x,-.5) -- (\x,5);
\filldraw (\x,\y) circle (1pt);
}}
\filldraw (.5,1.2) circle (1pt);
\node[below left] at (0,0) {$(0,0)$};
\node[left] at (-.5,2.25) {$2\beta n^{-\alpha/2}$};
\draw[decoration={brace},decorate]  (-.4,1.55) -- (-.4,2.95);
\draw[white,line width=0.7pt] (0,.7) -- (0,1.25);
\node at (0,.92) {$\tfrac{2\beta}{n^{\alpha/2}}S$};
\node[below right] at (4.9,5.7) {$A$};
\node[above left] at (3.1,2.95) {$z^*$};
\node at (3.7,3.3) {\footnotesize $\null ^\times$};
\node[above] at (3.7,3.3) {$\lambda$};
\node at (.7,.3) {\footnotesize $\null^\times$};
\node[above right] at (.4,.27) {$\lambda-z^\ast$};
%\draw[gray, line width=0.5pt, dotted] (2.5,0) -- (4,0);
%\filldraw [gray] (0,2.2) circle (1pt);
%\node[below left] at (0,0) {0};
%\node[below] at (1,0) {$\tfrac1{nR^{\frac2m}}$};
\end{tikzpicture}
\caption{The distance of $\lambda$ to random grid $A$ is illustrated by the gray doubled arrow. This is equal to the distance of the rescaled random shift $2\beta n^{\alpha/2}S$ to $\lambda-z^\ast$ (measured in the quotient space), where $z^\ast$ is the reference corner of the box containing $\lambda$.}
\label{fig:grid}
\end{figure}
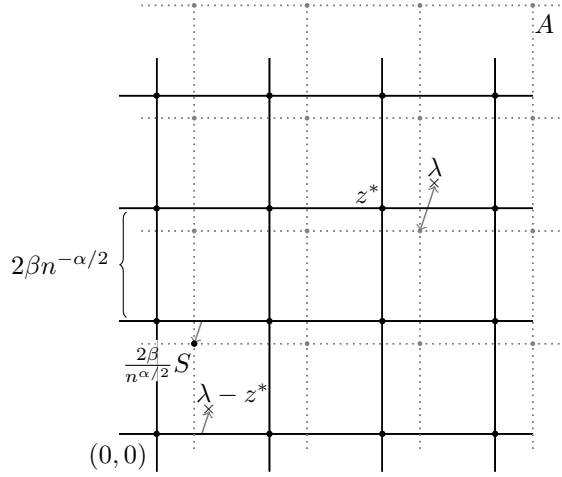

From now on we will restrict ourselves to this event $\tilde\Omega_\ast$. Rewrite \eqref{eq:fixedlambda} as
\begin{align*}
\sum_{i=1}^{\lceil n^{\alpha}\rceil} \int_{K_i}\Delta f(z)\log\abs{\lambda-z}-\Delta f(z_i)\log\abs{\lambda-z_i} d\lebesgue(z),
\end{align*}
where we denoted the boxes with corner $z_i$ by $K_i=z_i+[0,2\beta n^{-\alpha/2}]^2$. Adding and substracting $\Delta f(z_i)\log\abs{\lambda -z}$, we obtain one error of order
 \begin{align}\label{eq:MVT}
 \sum_{i=1}^{\lceil n^{\alpha}\rceil} \int_{K_i}(\Delta f(z)-\Delta f(z_i))\log\abs{\lambda-z}d\lebesgue(z)=\mathcal O(\norm{\nabla \Delta  f}_\infty n^{-\alpha/2}),
\end{align} 
where we used the mean value theorem and local integrability of $\log$ in $\IC$. The second term can be bounded by
\begin{align*}
\Bigg\lvert\sum_{i=1}^{\lceil n^{\alpha}\rceil} &\int_{K_i}\Delta f(z_i)\left(\log\abs{\lambda-z}-\log\abs{\lambda-z_i}\right)d\lebesgue(z)\Bigg\rvert\\
 &\leq\norm{\Delta f}_\infty\bigg(\sum_{i:\abs{z_i-\lambda}\ge n^{-\frac{\alpha}4}}+\sum_{i:\abs{z_i-\lambda}< n^{-\frac{\alpha}4}}\bigg)\int_{K_i}\abs{\log\abs{\lambda-z}-\log\abs{\lambda-z_i}}d\lebesgue(z).
\end{align*}
Applying the mean value theorem for $\log$, yields a bound of order $\mathcal O (n^{-\alpha/4})$ for the first sum. The second sum can be bounded by performing the integration in polar coordinates
\begin{align*}
 \int_0^{2n^{-\alpha/4}} r\big(\log r +\sup_{z_i\in A}\log \abs{\lambda-z_i}\big)dr=\mathcal O (n^{-\alpha/2}\log^2 n),
\end{align*}
where we finally used the event $\tilde\Omega_\ast$. Putting all estimates together proves the first claim.

The bound for $U_n$ follows from the choice of $A$ and a trivial upper bound on the spectral radius $\abs\lambda_{\max } $ of $\mathbf W$. On the one hand $\abs\lambda_{\max } $ is bounded by the largest singular value $s_{\max } $ and on the other hand we have
\begin{align*}
 \IP(s_{\max }\geq n^{\log n})\leq \IE\norm{(\mathbf W-z)}^2n^{-2\log n}%\leq \frac 1 {n^{2\log n+1}} \sum_{ij}^n\IE\abs{\mathbf W_{ij}}^2
 \lesssim n^{-2\log n+1},
\end{align*}
similar to \eqref{eq:s_max}. Therefore on the event $\Omega_\ast=\tilde\Omega_\ast\cap\{s_{\max}\le n^{\log n}\}$ with probability at least $1-\mathcal O(n^{-\log n})$, we have \eqref{eq:gridapprox} and
\begin{align*}
 \sup_{z_i\in A}\abs{U_n(z_i)}&\le \frac 1 n\sum_{j=1}^ n\sup_{z_i\in A}\big|\log\abs{\lambda_j-z_i}\big|\\
 &\lesssim (\log n+1+\alpha)\log n+\log\big|\abs\lambda_{\max }+5\big|=\mathcal O(\log^2 n).
\end{align*}
\end{proof}

Finally, we turn to the

\begin{proof}[Proof of Theorem \ref{thm:rateofconv}]
First, note that we only need to consider $\tau<1$. In order to restrict ourselves to a bounded region, say $V=B_7(0)$, we separate
\begin{align}\label{eq:first}
&\sup_{\substack{B_R(z_0)\subseteq \IC\setminus B_{1+\tau}\\ \text{or }B_R(z_0)\subseteq B_{1-\tau}}}\abs{ (\mu_n^m-\mu_\infty^m)(B_R(z_0))}\nonumber \\ \le &\sup_{\substack{B_R(z_0)\subseteq \IC\setminus B_{1+\tau}\\ \text{or }B_R(z_0)\subseteq B_{1-\tau}}}\abs{ (\mu_n^m-\mu_\infty^m)(B_R(z_0)\cap V)} +\mu_n^m(V^c).
 \end{align} 
%By covering $V^c$ with three sufficiently large balls, $\mu_n^m(V^c)=(\mu_n^m-\mu_\infty^m)(V^c)$ is again bounded by $3$ times (say) the left hand side of \eqref{eq:first}, hence we only have to estimate this first term. 

Fix some $\tau,R>0, z_0\in\IC$ such that $B_R(z_0)\subseteq B_{1-\tau}\cup B_{1+\tau}^c$. 

Let $\phi\in\mathcal{C}^\infty(\IR)$ be nonnegative with $\supp \phi \subseteq [-1,1]$ and $\int\phi=1$, and define $\phi_a(\rho)=a\phi(a\rho)$ for some large $n$-dependent parameter $a>1$ to be determined later. We define the cutoff to $V^c$ by 
\begin{align*}
 f_0(z)= \left(\eins_{(7-1/a,\infty)}\ast\phi_{a}\right)(\abs{z})\ge\eins_{V^c}(z) .
\end{align*}
Moreover, we mollify the indicator function $\eins_{B_R(z_0)\cap V}$ appearing in \eqref{eq:first} via the approximation
\begin{align*}
 f_1(z):&=\left(\eins_{(-\infty, R-1/a]}\ast\phi_{a}\right)(\abs{z-z_0} )\cdot \left(\eins_{(-\infty,7-1/a]}\ast\phi_{a}\right)(\abs{z})\\
 &\leq\eins_{B_R(z_0)\cap V}(z)\\
 &\leq\left(\eins_{(-\infty,R+1/a]}\ast\phi_a\right)(\abs{z-z_0} )\cdot \left(\eins_{(-\infty,7+1/a]}\ast\phi_{a}\right)(\abs{z} )=:f_2( z),
\end{align*}
where we choose $f_1\equiv 0$ if $R\le 2/a$ for smoothness reasons. 

We apply $f_1\leq \eins_{B_R(z_0)\cap V}$ and integration by parts to $\tilde f_1:z\mapsto f_1(z^m)$ as was explained in \eqref{eq:LocalLawIdentity} to obtain
\begin{align}\label{eq:3terms}
&\mu_n^m(B_R(z_0)\cap V)\geq\int f_1d\mu_n^m=-\frac{1}{2\pi}\int(\Delta \tilde f_1) U_n d\lebesgue\nonumber\\
&=-\frac{1}{2\pi}\int\Delta \tilde f_1 (U_n-U_\infty) d\lebesgue-\int(\eins_{B_R(z_0)\cap V}-f_1)d\mu_\infty^m+\int \eins_{B_R(z_0)\cap V}d\mu_\infty^m.
\end{align}
Analogous upper bounds hold for $f_0$ and $f_2$. A rough estimate of the error of approximation yields for the second term 
\begin{align}\label{eq:MaxMeasure}
 \int(\eins_{B_R(z_0)\cap V}-f_1)d\mu_\infty^m&\leq  \mu_\infty^m\left(z\in \IC: R-2/a\leq \abs{z-z_0 }\leq R\right) .
\end{align}
Due to the radial monotonicity of $\mu_\infty ^m$'s density, this value increases by bending segments of the given annulus of width $2/a$ into straight rectangles $[-1,1]\times[-4/a,4/a]$ for some constant $c>0$. Let us provide some more details on how a less technical, non-continuous bending procedure can be achieved: 

The distribution of the annulus \eqref{eq:MaxMeasure} is bounded by 10 times the maximal probability of an annulus segment with outer circumferential length $2\pi R/10\wedge 2$. By radial monotonicity of $\mu_\infty ^m$, this particular "worst case" $1/10$-th segment $S$ is given by one that is closest to the origin and by radial symmetry we assume it to be symmetric with respect to the $x$-axis. We want to bend (or rather project) $S$ onto the $y$-axis. Note that the angle between the $y$-axis and any tangent to the circumference line of the segment $S$ is always bounded by $\pi/10$. Split our segment $S$ in at most $ca$ many pieces each of outer circumferential length $10/a$ and shift them in $x$-direction until they touch the $y$-axis (i.e. "projecting" the pieces). This transformation only increases the probability since we may only decrease the radial part of all points in each piece. Now since the mentioned angle is bounded, at most a constant number of the shifted segments overlap. One can show that this constant is actually equal to 2, or in other words non-neighboring shifted pieces are disjoint. Therefore, the total probability in \eqref{eq:MaxMeasure} is bounded by $20\mu_\infty^m([-1,1]\times [10/a,10/a])$.

The density is bounded in the case of $m=1$ and hence the term in \eqref{eq:MaxMeasure} is of order $\mathcal O (1/a)$. In general we can bound it by
\begin{align*}
\mu_\infty^m([-1,1]\times[-c/a,c/a])&\leq\frac 1 {\pi m}\int_{-1}^1\int_{-c/a}^{c/a}(x^2+y^2)^{1/m-1}dxdy\\
&\leq\frac {4c} {\pi m a}\int_{c/(2a)}^1 x^{2/m-2}dx+\frac 2m\int_0^{4c/a}r^{2/m-1}dr\\
&=\frac 8 {\pi(m-2)}\left(\Big(\frac c {2a}\Big)^{2/m}-\Big(\frac c {2a}\Big) \right)+\Big(\frac{4c} {a}\Big)^{2/m}\lesssim a^{-2/m},
\end{align*}
where the equation only holds for $m>2$. For $m=2$ we get 
\begin{align*}
  &\mu_\infty^2([-1,1]\times[-c/a,c/a])\leq\frac 1 {2\pi}\int_{-1}^1\int_{-c/a}^{c/a}(x^2+y^2)^{-1/2}dxdy\\ 
  &=\frac {c}{\pi a}\log\left(\frac{\sqrt{a^2+c^2}+a}{\sqrt{a^2+c^2}-a}\right)+\frac 2 \pi\log\left(\sqrt{1+c^2/a^2}+c/a\right)\\
  &\sim a^{-1}\log a
\end{align*}
and we see that the $\log$-term appears naturally. Define the error function \begin{align}\label{eq:RingeRest}
\tilde h_m(a)=\begin{cases}
        \mathcal{O}(a^{-1})&\text{ for }m=1,\\
        \mathcal{O}(a^{-1}\log a)&\text{ for }m=2,\\
        \mathcal{O}(a^{-2/m})&\text{ for }m\ge 3.\\
        \end{cases}
\end{align}
Let us continue to estimate the first term of \eqref{eq:3terms} by using our random grid approximation Lemma \ref{lem:gridapprox}. Let $\beta=7$ and $S$ be a random variable, independent of $\mathbf X$ and uniformly distributed on $[0,1]^2$. Conditioned on $\mathbf X$, we have with overwhelming probability
\begin{align*}
 \int\Delta \tilde f_1(z) U_n(z) d\lebesgue(z)-\frac{(2\beta)^2}{n^\alpha}\sum_{i=1}^{\lceil n^{\alpha}\rceil}\Delta \tilde f(z_i)U_n(z_i)\\
 =\mathcal O\big(\big(\lVert\nabla \Delta \tilde f_1 \rVert_\infty+\lVert\Delta \tilde f_1 \rVert_\infty \log^2 n \big)n^{-\alpha/2}\big)
\end{align*}
Due to our explicit choice of functions $f_1$ and $f_2$ as product of shifted radial symmetric functions, the partial derivatives become fairly simple. Each derivative that hits one of the $\phi_a$ produces a factor of $a$, more precisely any $k$-th directional derivative satisfies $\lVert\partial^{(k)} f_1(z)\rVert_\infty\lesssim a^k$. This estimate, again, is independent on the choice of the ball $B_R(z_0)$.

Together with the Riemann approximation \eqref{eq:RiemannApprox}, we conclude that for any matrix $\mathbf X$ we have with overwhelming probability
\begin{align*}
 \sup_{B}\Big\lvert\int\Delta \tilde f_1(z) (U_n(z)-U_\infty(z)) d\lebesgue(z)&-\frac{(2\beta)^2}{n^\alpha}\sum_{i=1}^{\lceil n^{\alpha}\rceil}\Delta \tilde f_1(z_i)\left(U_n(z_i)-U_\infty(z_i)\right)\Big\rvert\\
&=\mathcal O\big((a^3+a^2\log^2 n) n^{-\alpha/2}\big),
\end{align*}
where the supremum runs over all choices of $B\subseteq B_{1-\tau}\cup B_{1+\tau}^c$. Since we will always choose $a\lesssim n$ (actually we will make it even smaller, cf. \eqref{eq:choosea}), it is possible to freely choose $\alpha>0$ sufficiently big such that the error is arbitrarily small. For instance $\alpha =13$ is more than enough to ensure that the error is of order $\mathcal O(n^{-1})$. It should be emphasized that still no randomness of $\mathbf X$ has been used and the only randomness is the shifted grid. Combining the previous steps yields
\begin{align*}%\label{eq:Losing}
(\mu_n^m-\mu_\infty^m)(B\cap V)\geq -\frac{(2\beta)^2}{2\pi n^\alpha}\sum_{i=1}^{\lceil n^{\alpha}\rceil}\Delta \tilde f_1(z_i)\Big(U_n(z_i)-U_\infty(z_i)\Big)-\tilde h_m(a)-\mathcal O (n^{-1})
\end{align*}
uniformly in $B\subseteq B_{1-\tau}\cup B_{1+\tau}^c$ with overwhelming probability. Noting $\mu_n^m(B_R(z_0)\cap V)\leq\int f_2d\mu_n^m$ and taking the same route for $f_2$ as for $f_1$, we obtain the same upper bound
\begin{align*}
(\mu_n^m-\mu_\infty^m)(B\cap V)\leq -\frac{(2\beta)^2}{2\pi n^\alpha}\sum_{i=1}^{\lceil n^{\alpha}\rceil}\Delta \tilde f_2(z_i)\Big(U_n(z_i)-U_\infty(z_i)\Big)+\tilde h_m(a)+\mathcal O (n^{-1}).
\end{align*}
In the same way, the bound holds for $\mu_n^m(V^c)$ as well. Finally we use the randomness of $\mathbf X$ by applying Proposition \ref{prop:ConcLogPot}. Conditioning on $S$, i.e. freezing the lattice points $z_i$, we obtain for any $Q>0$
\begin{align*}
 \IP\left( \abs{ U_n(z_i)-U_\infty(z_i)}\geq c \frac{\log^4(n)}{n}\Big\vert S\right)\leq n^{-Q-\alpha}
\end{align*}
for each $i=1,\dots,\lceil n^{\alpha}\rceil$. By the union bound this implies that with probability at least $1-n^{-Q}$ the logarithmic potentials concentrate like $U_n(z_i)-U_\infty(z_i)=\mathcal O( \log^4 n /n)$ simultaneously at all lattice points. Therefore, for $k=0,1,2$,
\begin{align*}
 \frac{(2\beta)^2}{2\pi n^\alpha}\sum_{i=1}^{\lceil n^{\alpha}\rceil}\Delta \tilde f_k(z_i)\Big(U_n(z_i)-U_\infty(z_i)\Big)&\lesssim \frac{(2\beta)^2\log ^4n}{n^{1+\alpha}}\sum_{i=1}^{\lceil n^{\alpha}\rceil}\big\lvert\Delta \tilde f_k(z_i)\big\rvert\\
 &=\frac{\log ^4 n}{n}\big\lVert\Delta \tilde f_k\big\rVert_{L^1}+\mathcal O \Big(a^3\frac{\log^4 n}{n^{\alpha/2+1}}\Big)
\end{align*}
where the integral of the $a^3$-Lipschitz function $\lvert\Delta \tilde f_k\rvert$ has been approximated by its Riemann sum. Write $\Delta=4\bar\partial\partial$ in terms of the Wirtinger derivatives $\partial=\frac12(\partial_x-i\partial_y)$ and $\bar\partial =\frac12(\partial_x+i\partial_y)$. Since $g(z)=z^m$ is holomorphic, i.e. $\bar\partial g=0$, we obtain by applying the chain rule and changing variables from $z$ to $g(z)$
\begin{align*}
 \lVert\Delta  \tilde f_k\big\rVert_{L^1}= \lVert4\bar\partial\partial  (f_k\circ g)\big\rVert_{L^1}=4\lVert(\bar\partial\partial  f_k)\circ g\cdot \bar\partial\bar g\cdot\partial g )\big\rVert_{L^1}=\lVert\Delta  f_k\big\rVert_{L^1}.
\end{align*}
%\footnote{Note that the $L^1$ norm of the Laplacian is exactly one part of the constant appearing in local circular laws.}  
Since $\Delta  f_k\lesssim a^2$ and has support on an area of order $a^{-1}$, we have
\begin{align*}
 \sup_B\big\lVert\Delta  f_k\big\rVert_{L^1}\lesssim a.
\end{align*}
So overall we have proven that for all $Q$ there exists a constant $c>0$ such that with probability $1-n^{-Q}$ we have
\begin{align}\label{eq:choosea}
 \sup_{B \subseteq B_{1-\tau}\cup B_{1+\tau}^c}\abs{(\mu^m_n-\mu^m_\infty)(B)}\leq c a\frac{\log^4 n}{n} +\tilde h_m(a) +\mathcal O (n^{-1}).
\end{align}
Optimizing in $a$ yields $a=\sqrt n/\log^2n$ for $m=1$, as well as $h_2(n)=\log^3n/\sqrt n$. The asymptotic $h_m(n)$ for higher $m$ follows from choosing $a=n^{m/m+2}\log^{-4m/(m+2)}n$.

\end{proof}
 
In the proof we have seen that the maximal error for the limiting distribution $\mu^m_\infty$ is by balls, which touch the origin (technically these balls of growing size are not even admissible here). This yields the non-optimal rate in Theorem \ref{thm:rateofconv} if we do not exclude the origin. Having Theorem \ref{thm:ProductMeanRate} in mind however, we expect the maximizing ball to appear roughly at $B_1(0)$, where the error would be optimal again. 

\begin{proof}[Proof of Corollary \ref{cor:rateofconv}]
As long as the origin is avoided by a fixed distance $\tau$, the density of $\mu_{\infty}^m$ is bounded as in the case of $m=1$. Therefore, the only adjustment to the previous proof of Theorem \ref{thm:rateofconv} is that the error term in \eqref{eq:MaxMeasure} is now also given by $\tilde h_m(a)=\mathcal O (a^{-1})$. Since the remaining part of the proof remains untouched, Corollary \ref{cor:rateofconv} follows.
\end{proof}

\section*{Acknowledgements}
Financial support by the German Research Foundation (DFG) through the IRTG 2235 is gratefully acknowledged. The author would like to thank Thorsten Neuschel and Friedrich G\"otze for valuable suggestions, Mario Kieburg for helpful discussions regarding the saddle-point analysis and the referees for very useful feedback and remarks.

%\bibliography{bibliography_EJP}%copy bib into same folder! and comment bibliography and biblatex from the beginning!
%\bibliographystyle{alpha} 
%\printbibliography
\end{document}